\documentclass{amsart}

\usepackage{amssymb,amsfonts}
\usepackage[all,arc]{xy}
\usepackage{enumerate}
\usepackage{tikz-cd}
\usepackage{mathrsfs}

\usepackage{hyperref}  
\hypersetup{%
	bookmarksnumbered=true,%
	colorlinks=true,%
	linkcolor=blue,%
	citecolor=blue,%
	filecolor=blue,%
	menucolor=blue,%
	urlcolor=blue,%
	pdfnewwindow=true,%
	pdfstartview=FitBH}

\def\makeautorefname#1#2{\expandafter\def\csname#1autorefname\endcsname{#2}}
\def\equationautorefname~#1\null{(#1)\null}
\makeautorefname{footnote}{footnote}%
\makeautorefname{item}{item}%
\makeautorefname{figure}{Figure}%
\makeautorefname{table}{Table}%
\makeautorefname{part}{Part}%
\makeautorefname{appendix}{Appendix}%
\makeautorefname{chapter}{Chapter}%
\makeautorefname{section}{Section}%
\makeautorefname{subsection}{Section}%
\makeautorefname{subsubsection}{Section}%
\makeautorefname{theorem}{Theorem}%
\makeautorefname{thm}{Theorem}%
\makeautorefname{sta}{Statement}%
\makeautorefname{cor}{Corollary}%
\makeautorefname{lem}{Lemma}%
\makeautorefname{prop}{Proposition}%
\makeautorefname{pro}{Property}
\makeautorefname{conj}{Conjecture}%
\makeautorefname{conv}{Convention}%
\makeautorefname{defn}{Definition}%
\makeautorefname{notn}{Notation}
\makeautorefname{notns}{Notations}
\makeautorefname{rem}{Remark}%
\makeautorefname{rems}{Remarks}%
\makeautorefname{quest}{Question}%
\makeautorefname{exmp}{Example}%
\makeautorefname{ax}{Axiom}%
\makeautorefname{cla}{Claim}%
\makeautorefname{ass}{Assumption}%
\makeautorefname{asses}{Assumptions}%
\makeautorefname{con}{Construction}%
\makeautorefname{prob}{Problem}%
\makeautorefname{warn}{Warning}%
\makeautorefname{obs}{Observation}%

\newtheorem{thm}{Theorem}[section]

\newtheorem{prop}{Proposition}[section]
\newtheorem{lem}{Lemma}[section]
\newtheorem{cla}{Claim}[section]

\theoremstyle{definition}
\newtheorem{defn}{Definition}[section]

\newtheorem{rem}{Remark}[section]

\newcommand{\sharecounter}[1]{%
	\expandafter\let\csname c@#1\expandafter\endcsname\csname c@thm\endcsname}
\sharecounter{obs}
\sharecounter{cor}
\sharecounter{prop}
\sharecounter{lem}
\sharecounter{prob}
\sharecounter{con}
\sharecounter{conj}
\sharecounter{defn}
\sharecounter{notn}
\sharecounter{notns}
\sharecounter{exmp}
\sharecounter{ax}
\sharecounter{pro}
\sharecounter{ass}
\sharecounter{warn}
\sharecounter{rem}
\sharecounter{conv}
\sharecounter{sch}
\sharecounter{equation}

\numberwithin{equation}{section}

\newcommand{\C}{\mathbb{C}}

\newcommand{\Q}{\mathbb{Q}}

\newcommand{\mtbb}{\mathbb}
\newcommand{\mtfk}{\mathfrak}
\newcommand{\mtcl}{\mathcal}

\newcommand{\K}{\mathbb{K}}

\newcommand{\N}{\mathbb{N}}
\renewcommand{\O}{\mathcal{O}}
\author{Haoran Zhu}
\email{zhu-hr22@mails.tsinghua.edu.cn}
\address{Qiuzhen College, Tsinghua University, Beijing, China}
\usepackage{amssymb}
\usepackage{amscd}
\usepackage{mathabx}
\usepackage{amsmath}
\usepackage{enumerate}
\usepackage{exscale}
\usepackage{relsize}
\usepackage{paralist}
\newcommand{\subsectionnotoc}[1]{%
	\par
	\addvspace{.5\linespacing plus .7\linespacing}%
	\noindent{\normalfont\bfseries #1.}\hspace{.5em}%
}
\title{Boundedness of Polarized Calabi--Yau Fibrations}
\begin{document}
	\begin{abstract}
		We investigate the boundedness problem for log Calabi–Yau fibrations whose bases and general fibers are bounded. We improve the results in \cite{jjz25} from birational boundedness to actual boundedness, which confirms a conjecture of Birkar and Hacon when the boundary coefficients belong to a fixed finite rational set. We also get a new proof for the birational boundedness result in \cite{jjz25} which is substantially more conceptual.
	\end{abstract}
	\maketitle
	\tableofcontents
	\section{Introduction}
	\paragraph{Convention.}
	We work over $\C$ throughout this paper.
	
	In this paper, we investigate the boundedness problem for polarized Calabi--Yau fibrations. Jiang--Jiao--Zhu proved that the couples in Theorem~\ref{mainthm} are log bounded in codimension one~\cite[Theorem~1.4]{jjz25}, and proved log boundedness under the additional assumption that the irregularity of the general fiber vanishes~\cite[Theorem~1.7]{jjz25}.
	\begin{defn}
		A polarized log Calabi--Yau fibration \( f : (X, B), A \to (Z, H) \) consists of
		\begin{itemize}
			\item a normal quasi-projective pair \((X, B)\),
			\item a projective fibration \( f : X \to Z \) such that \( K_X + B \sim_{\mathbb{Q}} f^{\ast}N \) for some \(\mathbb{Q}\)-Cartier divisor \( N \) on \( Z \),
			\item an integral divisor \( A \) on \( X \) that is ample over \( Z \), and
			\item a very ample divisor \( H \geq 0 \) on \( Z \) such that \( H - N \) is ample.
		\end{itemize}
		We call \( f : (X, B), A \to (Z, H) \) a weak polarized log Calabi--Yau fibration if \( H - N \) is only assumed to be pseudo-effective.
	\end{defn}
	\begin{defn}
		Let \( d \in \mathbb{Z}^{>0} \), \( v, r, \varepsilon \in \mathbb{Q}^{>0} \), and let \( \Phi \subset [0, 1] \cap \mathbb{Q} \) be a DCC set.
		\begin{enumerate}
			\item[(1)] A \emph{(weak) \((d, r, \varepsilon)\)-polarized log Calabi--Yau fibration} is a (weak) polarized log Calabi--Yau fibration \( f : (X, B), A \to (Z, H) \) satisfying
			\begin{itemize}
				\item \( (X, B) \) is a projective \( \varepsilon \)-lc pair of dimension \( d \), and
				\item \( H^{\dim Z} \leq r \).
			\end{itemize}

			\item[(2)] If, additionally,
			\begin{itemize}
				\item \( \operatorname{vol}_F(A|_F) \leq v \), where \( F \) is a general fiber of \( f : X \to Z \),
			\end{itemize}
			then we call \( f : (X, B), A \to (Z, H) \) a \emph{(weak) \((d, v, r, \varepsilon)\)-polarized log Calabi--Yau fibration}.
			
			\item[(3)] Furthermore, if
			\begin{itemize}
				\item the coefficients of \( B \) belong to \( \Phi \),
			\end{itemize}
			then we refer to \( f : (X, B), A \to (Z, H) \) as a \emph{(weak) \((d, \Phi, v, r, \varepsilon)\)-polarized log Calabi--Yau fibration}.
		\end{enumerate}
	\end{defn}
	\begin{thm}[Main theorem]\label{mainthm}
		Let \( d \in \mathbb{N} \), \( v, r, \varepsilon \in \mathbb{Q}^{>0} \) and \( \Phi \subset [0, 1] \cap \mathbb{Q} \) be a finite set. Consider the set of all weak \((d, \Phi, v, r, \varepsilon)\)-polarized log Calabi--Yau fibrations \( f : (X, B), A \to (Z, H) \). Then the couples
		\[
		\left(X,\operatorname{Supp}(B+f^{\ast}H)\right)
		\]
		form a log bounded family.
	\end{thm}
	Compared with~\cite[Theorem~1.7]{jjz25}, Theorem~\ref{mainthm} removes the irregularity assumption and answers~\cite[Question~1.6]{jjz25} when the coefficients of $B$ belong to a fixed finite rational set. The proof below does not use the boundedness theorem of~\cite{jjz25}. Its starting point is instead the curve case. For a general  curve $C\subset Z$ cut by $H^{\dim Z-1}$, boundedness over $C$ controls the intersection of $C$ with the divisorial part of the locus where the fibration fails to define a family of traditional stable minimal models. This produces a reduced divisor $D$ of bounded degree such that $(Z,D)$ is log bounded, $U:=Z\setminus D$ is smooth, and the fibration gives a traditional-stable-minimal-model family over a big open subset $V\subset U$. The associated coarse moduli map is bounded by its restrictions to such curves. After enlarging $D$ within a bounded linear system, this coarse map is regular on $U$, and~\cite[Corollary~3.5]{kovacs2009boundednessfamiliescanonicallypolarized} extends the lift $V\to\mathfrak M$ uniquely to $U\to\mathfrak M$. Thus not only the underlying families but also the polarizations are bounded over $U$. 
	Then the boundedness follows from standard arguments which also appear in \cite{jjz25}.
	
	The following curve result is the principal external boundedness input.
	\begin{thm}[Curve case]\label{curvecase}
		Let \( d \in \mathbb{N} \), \( v, r, \varepsilon \in \mathbb{Q}^{>0} \), and let \( \Phi \subset [0, 1] \cap \mathbb{Q} \) be a finite set. There exist a bounded family of couples $\mtcl P$ and a number $\lambda$, depending only on $d,\Phi,v,r,\varepsilon$, with the following property. Let $f:(X,B),A\to (C,H)$ be a weak $(d,\Phi,v,r,\varepsilon)$-polarized log Calabi--Yau fibration over a curve $C$. Then $(X,B)\to C$ belongs to $\mtcl P$, and there exists $A^{\prime}\sim_{C} A$ such that $\operatorname{vol}_{X}(A^{\prime})\leq \lambda$, $A^{\prime}$ is ample on $X$, and $A^{\prime}-f^{\ast}H$ is pseudo-effective. Moreover, there exists a positive integer $l$, depending only on $d,v,r,\varepsilon,\Phi$, such that $lA^{\prime}$ is very ample on $X$.
	\end{thm}
	\begin{proof}
		Write $K_X+B\sim_{\Q}f^{\ast}N$. The canonical bundle formula and the pseudo-effectivity of $H-N$ give
		\[
		-2\leq\deg N\leq\deg H\leq r.
		\]
		The proof of~\cite[Proposition~3.4]{Hashizume_2025} constructs a divisor $A'\sim_C A$ which is ample with bounded volume. In that construction the numerical estimate for $\operatorname{vol}(A')$ depends on $u=\deg N$ only through an upper bound for $|u|$. The boundedness of the denominator of $u$ is not used in the construction of $A'$.
		
		In our setting the pairs are already $\varepsilon$-lc and the coefficients of $B$ lie in the fixed finite set $\Phi$. Thus the boundedness argument in~\cite[Theorem~3.6]{Hashizume_2025} implies that $(X,B)$ forms a bounded family and a bounded multiple of $A^{\prime}$ is Cartier. Finally, the effective very ampleness give a uniform $l$ for which $lA'$ is very ample.
	\end{proof}
	The arguments in unpublished notes of Birkar--Hacon also essentially prove Theorem~\ref{curvecase}.
	
	\subsectionnotoc{Acknowledgements}
	The author would like to thank his advisor, Caucher Birkar for his encouragement and constant support, and for suggesting the author to consider whether techniques from the geometric Shafarevich program could be used to give a more conceptual proof of results in \cite{jjz25}. This suggestion initiated the present work.
	The author would also like to thank Caucher Birkar, Bingyi Chen, Junpeng Jiao, Xiaowei Jiang, Xintong Jiang, Behrouz Taji, and Fulin Xu for helpful discussions. He would also like to thank Junpeng Jiao, Xiaowei Jiang, Xintong Jiang, and Fulin Xu for pointing out mistakes in previous versions of this manuscript.
	
	\subsectionnotoc{Disclosure of AI Use}
	The main parts of this work were completed independently by the author without reliance on AI. ChatGPT was used to assist in proofreading early drafts, fixing technical mistakes, and polishing the elaboration. The argument in the proof of Lemma~2.15 and the argument in the Appendix A were suggested by ChatGPT.

	\section{Preliminary Results}
	\subsection{Notation and conventions}
	We use the following conventions throughout the paper.
	\begin{enumerate}
		\item A projective morphism $f:X\to Z$ between normal quasi-projective varieties is called a \emph{contraction} if $f_{\ast}\O_X=\O_Z$. Thus $f$ is surjective with connected fibers. The word \emph{fibration} will always mean a projective contraction.
		
		\item Unless otherwise stated, every curve is assumed to be normal, or equivalently smooth over $\C$.
		
		\item An \emph{integral divisor} is a $\Q$-Cartier Weil divisor with integer coefficients; it is not necessarily Cartier or effective.
		
		\item Following~\cite[Tag~056P]{stacks-project}, if $f:X\to S$ is a morphism, a \emph{relative effective Cartier divisor} on $X/S$ is an effective Cartier divisor $D\subset X$ such that $D\to S$ is flat.
		
		\item For a fibration $f:X\to Z$, we denote its generic fiber by $X_\eta$ and a general fiber by $X_g$. If $D$ is a divisor on $X$, we write $D_\eta:=D|_{X_\eta}$ and $D_g:=D|_{X_g}$ whenever these restrictions are defined.
		
		\item Let $f:X\to Z$ be a morphism and let $L,M$ be $\Q$-Cartier $\Q$-divisors on $X$. We write $L\sim_{\Q,Z}M$ if $L-M\sim_{\Q}f^{\ast}N$ for some $\Q$-Cartier $\Q$-divisor $N$ on $Z$. The analogous notation $\sim_Z$ is used for linear equivalence over $Z$.
		
		\item A prime divisor $E\subset X$ is \emph{horizontal} over $Z$ if $f(E)=Z$, and is \emph{vertical} otherwise. A prime divisor $E\subset X$ is called $f$-exceptional if $\operatorname{codim}_Z f(E)\geq2$. If $f:X\to Z$ is birational, then $\operatorname{Exc}(f)\subset X$ denotes the locus where $f$ is not an isomorphism, and
		\[
		\operatorname{Exc}_Z(f):=f\bigl(\operatorname{Exc}(f)\bigr)\subset Z
		\]
		denotes its image on $Z$. In particular, $\operatorname{Exc}(f)$ may have components of codimension greater than one and is not restricted to its divisorial part.
		
		\item If $\phi:X\dashrightarrow Y$ is a birational map and $D$ is a divisor on $X$, then $D_Y$ denotes the strict transform of $D$ on $Y$.
	\end{enumerate}
	\subsection{Strict transforms and linear equivalence}
	\begin{lem}[Linear equivalence under strict transform]\label{stricttransformlinearequiv}
		Let $f:X\to S,\ g:Y\to S$ be contractions between normal varieties. Let $\phi:X\dashrightarrow Y$ be a birational contraction over $S$. Let $D_1,D_2$ be $\Q$-Cartier $\Q$-divisors on $X$ whose strict transforms $(D_1)_Y,(D_2)_Y$ are $\Q$-Cartier. If
		\[
		D_1\sim_{\Q,S}D_2,
		\]
		then
		\[
		(D_1)_Y\sim_{\Q,S}(D_2)_Y.
		\]
		The same conclusion holds for linear equivalence over $S$.
	\end{lem}
	\begin{proof}
		Take a common resolution
		\[
		\xymatrix{
			& W\ar[dl]_p\ar[dr]^q &\\
			X\ar@{-->}[rr]^{\phi} &&Y.
		}
		\]
		After clearing denominators, write
		\[
		m(D_1-D_2)=\operatorname{div}_X(a)+f^{\ast}L,
		\]
		where $m>0$, $a\in\C(X)^{\ast}$, and $L$ is a Cartier divisor on $S$. Thus $p^{\ast}m(D_1-D_2)-\operatorname{div}_W(a)-p^{\ast}f^{\ast}L$ is $p$-exceptional. 
		Since $\phi$ is a birational contraction, it is also $q$-exceptional. Thus we have
		\[
		m\bigl((D_1)_Y-(D_2)_Y\bigr)
		=\operatorname{div}_Y(a)+g^{\ast}L.
		\]
		This proves the first assertion; the linear equivalence case is identical with $m=1$.
	\end{proof}
	\subsection{Bounded families and refinements of parameter spaces}
	\begin{defn}[Flat families with divisorial data]
		Let $T$ be a normal quasi-projective scheme. A \emph{flat family with
			divisorial data over $T$} consists of a flat projective morphism
		\[
		f:X\longrightarrow T
		\]
		together with an effective $\mathbb Q$-divisor
		\[
		B=\sum_{j=1}^m b_j D_j
		\]
		on $X$, where
		
		\begin{enumerate}
			\item $X$ is normal and $D_1,\ldots,D_m$ are distinct prime divisors on $X$;
			
			\item $0\leq b_j\leq 1$ for all $j$;
			\item $D_1,\ldots,D_m$ are faithfully flat with geometrically reduced fibers over $T$;
			\item no irreducible component of a fiber of $f$ is contained in
			$\operatorname{Supp}(B)$;
			
			\item for every closed point $t\in T$, the fiber $X_t$ is normal.
		\end{enumerate}
		
	\end{defn}
	\begin{defn}[Log bounded families]\label{logbounded}
		Let $\mtcl P$ be a collection of pairs $(X,B)$, where $B$ is an effective $\mathbb Q$-divisor. We say that $\mtcl P$ is \emph{log birationally bounded} (respectively, \emph{log bounded}, or \emph{log bounded in codimension one}) if there exist finitely many flat families with divisorial data
		\[
		f_i:(\mtcl X_i,\mtcl B_i)\longrightarrow T_i,
		\qquad 1\leq i\leq N,
		\]
		such that, for every $(X,B)\in\mtcl P$, there are an index $i$, a closed point $t\in T_i$, and a birational map
		\[
		\phi:(\mtcl X_i)_t\dashrightarrow X
		\]
		(respectively, an isomorphism, or a birational map that is an isomorphism in codimension one) for which
		\[
		\operatorname{Supp}(\phi_*^{-1}B)\cup
		\operatorname{Exc}_{\mathrm{div}}(\phi)
		\subseteq
		\operatorname{Supp}((\mtcl B_i)_t).
		\]
		Here $\operatorname{Exc}_{\mathrm{div}}(\phi)$ denotes the union of the prime divisors on $(\mtcl X_i)_t$ contracted by $\phi$. These definitions are equivalent to the standard ones.
	\end{defn}
	
	A collection of morphisms (respectively, rational maps) $\mathcal{F}$ is \textit{bounded} if there exist quasi-projective schemes $\mathcal{X}$ and $\mathcal{Z}$, together with projective morphisms $\mathcal{X} \stackrel{\phi}{\to} \mathcal{Z} \to T$ (respectively, a rational map $\varphi:\mtcl X\dashrightarrow\mtcl Z$ and a projective morphism $\mtcl Z\to T$), where $T$ is of finite type, such that every morphism $X\to Z$ (respectively, rational map $X\dashrightarrow Z$) in $\mathcal F$ is isomorphic to the restriction $\mathcal X_t\to\mathcal Z_t$ (respectively, $\mtcl X_t\dashrightarrow\mtcl Z_t$) for some closed point $t\in T$; in the rational-map case, $\varphi$ is required to restrict to a rational map on $\mtcl X_t$.
	
	The following lemmas record standard refinements of parameterizing families that will be used repeatedly. 
	\begin{lem}\label{hensel}
		Let
		\(
		f:(X,B)\longrightarrow T
		\)
		be a flat family with divisorial data, where $T$ is a quasi-projective normal
		variety. Then there exists a finite stratification
		\[
		T=\coprod_{i=1}^N S_i
		\]
		by locally closed normal subvarieties and finite \'etale morphisms
		\[
		\tau_i:T_i\longrightarrow S_i
		\]
		such that, writing
		\[
		(X_{T_i},B_{T_i}):=(X,B)\times_T T_i
		\]
		and
		\[
		B_{T_i}=\sum_{\beta} b_{\beta}D_{\beta,T_i},
		\]
		the following properties hold:
		
		\begin{itemize}
			\item $S_i,T_i$ are quasi-projective normal varieties;
			\item the morphisms
			\[
			f_i:X_{T_i}\longrightarrow T_i
			\]
			are faithfully flat and projective;
			
			\item each $D_{\beta,T_i}$ is a prime divisor on $X_{T_i}$;
			
			\item each morphism
			\[
			D_{\beta, T_i}\longrightarrow T_i
			\]
			is faithfully flat with geometrically integral fibers; and
			
			\item for every closed point $t\in T_i$, the fiber $(X_{T_i})_t$ is normal.

		\end{itemize}

	\end{lem}
	\begin{proof}

		By resolution of singularities and Noetherian induction, there exists a stratification \(
		T=\coprod_i S_i
		\)
		such that $S_i$ are quasi-projective normal varieties and for each $S_i$ there is a projective birational morphism
		
		\[
		\mu_i:(Y_i,\Omega_i)
		\longrightarrow
		(X,B)\times_T S_i
		\]
		
		satisfying:
		
		\begin{enumerate}
			\item for every closed point $s\in S_i$,
			\(
			\mu_s:(Y_i)_s\longrightarrow X_s
			\)
			is a birational morphism whose exceptional locus is divisorial with only horizontal components;
			
			\item the morphism
			\(
			(Y_i,\Omega_i)\longrightarrow S_i
			\)
			is log smooth and for every $s\in S_i$ the support of $(\Omega_i)_s$ contains the exceptional locus of $(\mu_i)_s$ and the support of the strict transform of $B_s$.
		\end{enumerate}
		
		By Claim 4.38.1 in~\cite{kovacsSingularitiesMinimalModel}, there exists a finite \'etale cover
		\(
		T_i\longrightarrow S_i
		\)
		such that $E_{\alpha,T_i}$ has connected fibers over $T_i$, where $(Y_{i,T_i},\Omega_{i,T_i})=(Y_i,\Omega_i)\times_{S_i}T_i$, $\Omega_{i,T_i}=\sum_{\alpha}E_{\alpha,T_i}$, and the $E_{\alpha,T_i}$ are prime divisors on $Y_{i,T_i}$. Because every $E_{\alpha,T_i}$ is smooth with connected fibers over $T_i$, its fibers are geometrically integral. Let $D_{\beta,T_i}$ denote an arbitrary component of $B_{T_i}$, and let $E_{\alpha,T_i}$ denote its strict transform in $Y_{i,T_i}$. By construction, $D_{\beta,T_i}$ is flat with geometrically reduced fibers over $T_i$. For every closed point $t\in T_i$, the fiber $(E_{\alpha,T_i})_t$ is a prime divisor containing the strict transform of $(D_{\beta,T_i})_t$. Thus $(D_{\beta,T_i})_t$ is irreducible and hence integral. Since $T_i$ is of finite type, all fibers of $D_{\beta,T_i}\to T_i$ are geometrically integral. The remaining assertions follow directly from the construction.
	\end{proof}
	\begin{lem}[Finite-coefficient refinement]\label{finitecoeffrefinement}
		Let $\mtcl P$ be a log bounded collection of pairs $(X,B)$. If the coefficients of all the divisors $B$ belong to a fixed finite subset $\Phi$ of $\mathbb Q_{\geq0}$, then there exist finitely many flat projective morphisms
		\[
		f_i:\mtcl X_i\longrightarrow T_i
		\]
		between normal quasi-projective varieties, with normal fibers, and effective $\mathbb Q$-divisors
		\[
		\mtcl B_i=\sum_{j=1}^{m_i}b_{ij}\mtcl D_{ij}
		\qquad (b_{ij}\in\Phi),
		\]
		such that every $\mtcl D_{ij}$ is a prime divisor and is faithfully flat over $T_i$ with geometrically integral fibers, and, for every $(X,B)\in\mtcl P$, there are $i,t$ and an isomorphism $X\simeq(\mtcl X_i)_t$ under which
		\[
		B=(\mtcl B_i)_t.
		\]
	\end{lem}
	\begin{proof}
		Choose finitely many families $f_i:(\mtcl X_i,\mtcl B_i)\to T_i$ realizing the log boundedness of $\mtcl P$. By Lemma~\ref{hensel}, we may assume that every component of $\mtcl B_i$ restricts to a prime divisor on each nonempty fiber of $f_i$. Assigning to $\operatorname{Supp}(\mtcl B_i)$ all possible coefficients from $\Phi$, including coefficient zero by omitting a component, gives only finitely many new families. These families have all the asserted properties.
	\end{proof}
	
	\subsection{Locally stable families}
	\begin{defn}[Relative Mumford divisor]
		Let \(f : X \to S\) be a flat morphism of finite type with \(S_2\) fibers of pure dimension \(d\). A subscheme \(D \subset X\) is a \emph{relative Mumford divisor} if there is an open set \(U \subset X\) such that
		\begin{itemize}
			\item \(\operatorname{codim}_{X_s}(X_s \setminus U_s) \ge 2\) for each \(s \in S\),
			\item \(D|_{U}\) is a relative Cartier divisor,
			\item \(D\) is the closure of \(D|_{U}\), and
			\item \(X_s\) is smooth at the generic points of \(D_s\) for every \(s \in S\), where $D_s$ is the closure of $D|_{U_s}$ in $X_s$.
		\end{itemize}
		By saying that \(D|_{U}\) is relatively Cartier, we mean that \(D|_{U}\) is a Cartier divisor on \(U\) whose support contains no irreducible component of any fiber \(U_s\).
		
		If \(D \subset X\) is a relative Mumford divisor for \(f : X \to S\) and \(T \to S\) is a morphism, then the \emph{divisorial pullback} \(D_T\) on \(X_T := X \times_S T\) is the relative Mumford divisor defined as the closure of the pullback of \(D|_{U}\) to \(U_T\). In particular, for each \(s \in S\), the divisor \(D_s = D|_{X_s}\) is the closure of \(D|_{U_s}\), namely the divisorial pullback of \(D\) to \(X_s\).
	\end{defn}
	\begin{defn}[Locally stable family]
		Following~\cite[Chapter~4]{book}, a \textbf{locally stable family} of pairs \((X, B) \to S\) over a reduced Noetherian scheme \(S\) consists of a flat morphism \(X \to S\) of finite type with \(S_2\) fibers and an effective \(\Q\)-divisor \(B\) on \(X\) satisfying
		\begin{itemize}
			\item each prime component of \(B\) is a relative Mumford divisor,
			\item \(K_{X/S} + B\) is \(\Q\)-Cartier, and
			\item \((X_s, B_s)\) is an slc pair for every point \(s \in S\).
		\end{itemize}
		If every fiber is klt, respectively lc, then we call \((X,B)\to S\) a \emph{locally stable family of klt pairs}, respectively a \emph{locally stable family of lc pairs}.
	\end{defn}
	We will need the following lemma. 
	\begin{lem}\label{fibtot}
		Let $f:(X,B)\to S$ be a locally stable family such that $S$ is smooth and the generic fiber of $f$ is normal. Then $X$ is normal and $(X,B)$ is lc. If we further suppose $(X_s,B_s)$ is klt for every $s\in S$, then $(X,B)$ is klt.   
	\end{lem}
	\begin{proof}
		By Theorem 4.54 in~\cite{book}, $(X,B)$ is slc and the conductor divisor does not contain any fiber. Since the generic fiber of $f$ is normal, the conductor divisor must be $0$. Thus $X$ is normal and $(X,B)$ is lc.
		
		Now suppose that $f$ is a family of klt pairs. We prove that $(X,B)$ is klt by induction on $d=\dim S$. The statement is immediate when $d=0$. Suppose that $d\geq 1$, and let $s\in S$ be an arbitrary closed point. Let $D\subset S$ be a smooth prime divisor passing through $s$, and let $f_D:(X_D,B_D)\to D$ be the restriction of $f$ over $D$. By the induction hypothesis, $(X_D,B_D)$ is klt. Thus $(X,B+f^{\ast}D)$ is plt near $f^{-1}D$. In particular, $(X,B)$ is klt near $f^{-1}s$.
	\end{proof}
	\begin{lem}\label{genericls}
		Let $f:(X,B)\to Z$ be a contraction between normal varieties, where $(X,B)$ is a klt pair. Then there exists a nonempty open subset $U\subset Z$ such that $f$ is locally stable with klt fibers over $U$.
	\end{lem}
	\begin{proof}
		Choose a birational morphism $\varphi:(Y,\Gamma)\to (X,B)$ such that $(Y,\Gamma)$ is a log smooth klt pair and $K_Y+\Gamma\geq \varphi^{\ast}(K_X+B)$. Let $g:(Y,\Gamma)\to Z$ be the composition of $\varphi$ and $f$. There exists a nonempty open subset $U\subset Z$ such that $U$ is smooth and $g$ is relatively log smooth over $U$. Let $f_U:(X_U,B_U)\to U$ and $g_U:(Y_U,\Gamma_U)\to U$ denote the restrictions of $f$ and $g$ over $U$.
		
		Let $p$ be an arbitrary closed point of $U$, and let $D$ be a reduced simple normal crossing divisor on $U$ such that $p$ is a stratum of $D$. By definition, $g_U$ is a locally stable family over $U$. Thus, by Theorem 4.54 in~\cite{book}, $(Y_U,\Gamma_U+g_U^{\ast}D)$ is lc. Since $K_Y+\Gamma\geq \varphi^{\ast}(K_X+B)$, the pair $(X_U,B_U+f_U^{\ast}D)$ is also lc. Another application of Theorem 4.54 in~\cite{book} shows that $f_U:(X_U,B_U)\to U$ is locally stable.
		
		Let $D_1$ be an irreducible component of $D$, and let $f_{D_1}:(X_{D_1},B_{D_1})\to D_1$ and $g_{D_1}:(Y_{D_1},\Gamma_{D_1})\to D_1$ be the restrictions of $f_U$ and $g_U$ over $D_1$. By Lemma~\ref{fibtot}, $(Y_{D_1},\Gamma_{D_1})$ is klt. By inversion of adjunction, $(Y_U,\Gamma_U+g_U^{\ast}D_1)$ is plt. Thus $(X_U,B_U+f_U^{\ast}D_1)$ is also plt, and adjunction implies that $(X_{D_1},B_{D_1})$ is klt. Moreover, by construction, $\varphi_{D_1}:(Y_{D_1},\Gamma_{D_1})\to X_{D_1}$ is a simultaneous log resolution of $(X_{D_1},B_{D_1})$ over $D_1$, with $K_{Y_{D_1}}+\Gamma_{D_1}\geq \varphi_{D_1}^{\ast}(K_{X_{D_1}}+B_{D_1})$. Replacing $(Z,D)$ successively by $(D_1,(D-D_1)|_{D_1})$ and repeating the argument shows that the fiber $(X_p,B_p)$ over $p=D_1\cap\cdots\cap D_r$ is klt.
	\end{proof}
	
	\subsection{Bounded log Calabi--Yau fibrations}
	\begin{lem}\label{hanjiang}
		Let $\mtcl P$ be a collection of data
		\[
		\pi:(X,B),L_X\longrightarrow(Z,L_Z),
		\]
		where $(X,B)$ is a projective pair, $L_X$ and $L_Z$ are very ample divisors, $\pi:X\to Z$ is a contraction, and $K_X+B\sim_{\Q,Z}0$. Suppose that there exist a positive integer $d$, positive numbers $r,\varepsilon$, and a finite set $I\subset\Q$ such that, for every datum in $\mtcl P$,
		\begin{itemize}
			\item $(X,B)$ is $\varepsilon$-lc of dimension $d$ and the coefficients of $B$ belong to $I$;
			\item $\operatorname{vol}_X(L_X)\leq r$; and
			\item $L_X-B_{\mathrm{red}}$ and $L_X-\pi^{\ast}L_Z$ are pseudo-effective.
		\end{itemize}
		Then  there exist finitely many projective families
		\[
		\begin{tikzcd}
			{(\mathcal X_i,\mathcal B_i)},\mtcl L_{\mtcl X_i}
			\arrow[r,"p_i"]
			\arrow[dr,"h_i"']
			&
			{(\mathcal Z_i,\mtcl L_{\mtcl Z_i})}
			\arrow[d,"g_i"]
			\\
			& T_i
		\end{tikzcd}
		\]
		such that
		\begin{itemize}
			\item $T_i$ and $\mathcal Z_i$ are normal quasi-projective varieties;
			\item $(\mathcal X_i,\mathcal B_i)$ is klt;
			\item $(\mathcal X_i,\mathcal B_i)\to T_i$ is locally stable with klt fibers;
			\item all three morphisms are projective and surjective;
			\item $\mathcal X_i,\mathcal B_i,\mtcl L_{\mtcl X_i},\mathcal Z_i$ are flat over $T_i$;
			\item every irreducible component of $\operatorname{Supp}(\mathcal B_i)$ is faithfully flat over $T_i$ with geometrically integral fibers;
			\item $K_{\mathcal X_i}+\mathcal B_i\sim_{\Q,\mathcal Z_i}0$;
			\item for every $\pi:(X,B),L_X\to (Z,L_Z)$ in $\mtcl P$, there exist some $i$ and $t\in T_i$ such that it is isomorphic to
			\[
			p_{i,t}:((\mtcl X_i)_t,(\mtcl B_i)_t),(\mtcl L_{\mtcl X_i})_t
			\longrightarrow
			((\mtcl Z_i)_t,(\mtcl L_{\mtcl Z_i})_t).
			\]
			Here we require the isomorphism to preserve the morphism, the pair, and the polarizations.
		\end{itemize}
	\end{lem}
	\begin{proof}
		The boundedness argument is the same as in~\cite[Lemma~2.8]{Han_2024}. Boundedness of Chow schemes gives finitely many projective families $h_i:\mathcal X_i\to T_i$, reduced divisors $\mathcal E_i$ on $\mathcal X_i$, and line bundles $\mtcl L_{\mtcl X_i},\,\mathcal L_i$ on $\mtcl X_i$ such that for every $(X,B),L_X\to(Z,L_Z)$ there exists $t\in T_i$, with
		\[
		X\simeq (\mtcl X_i)_{t},
		\quad
		\operatorname{Supp}(B)\subseteq\operatorname{Supp}((\mathcal E_i)_t),
		\qquad
		L_X\simeq(\mtcl L_{\mtcl X_i})_t,
		\qquad
		\pi^{\ast}L_Z\simeq(\mathcal L_i)_t.
		\]
		After finite stratifications and finite \'etale base changes, Lemma~\ref{hensel} allows us to write
		\[
		\operatorname{Supp}(\mathcal E_i)=\bigcup_{\beta=1}^{m_i}\mathcal D_{i,\beta},
		\]
		where every $\mathcal D_{i,\beta}$ is a prime divisor, faithfully flat over $T_i$, with geometrically integral fibers. Replacing $I$ by $I\cup\{0\}$ and taking the finitely many divisors
		\[
		\mathcal B_{i,\mathbf b}:=
		\sum_{\beta=1}^{m_i}b_\beta\mathcal D_{i,\beta},
		\qquad \mathbf b\in I^{m_i},
		\]
		we may assume that every relevant fiber boundary is exactly $(\mathcal B_i)_t=B$.
		
		By Noetherian induction and~\cite[Chapter~III, Corollary~12.9]{hartshorneAlgebraicGeometry1977}, after further stratification, the evaluation map
		\[
		h_i^{\ast}(h_i)_{\ast}\mathcal L_i\longrightarrow\mathcal L_i
		\]
		is surjective and commutes with base change. Thus $\mathcal L_i$ defines $p_i:\mathcal X_i\to\mathcal Z_i$ over $T_i$, together with a relatively ample line bundle $\mtcl L_{\mtcl Z_i}$ satisfying $p_i^{\ast}\mtcl L_{\mtcl Z_i}=\mathcal L_i$. The fibers recover $\pi:X\to Z$ and $L_Z$. By further stratification, using~\cite[Theorem~12.1.6 and Proposition~6.8.3]{PMIHES_1967__32__5_0}, we may assume that $T_i$ and $\mathcal Z_i$ are normal and that all the displayed data are flat over $T_i$.
		
		By~\cite[Proposition~2.4]{hacon2014boundednesslogcalabiyaupairs}, applied with $\varepsilon/2$, Lemma~\ref{genericls}, smooth stratification, and Noetherian induction, we may assume that each $T_i$ is smooth and
		\[
		(\mathcal X_i,\mathcal B_i)\longrightarrow T_i
		\]
		is locally stable with klt fibers. 
		By Lemma 4.52 in \cite{book}, after possibly passing to a stratification on $T_i$, we may assume $K_{\mtcl X_i}+\mtcl B_i\equiv_{\mtcl Z_i} 0$. Lemma~\ref{fibtot} and \cite[Corollary~1.6]{haconxu2016brepresentation} then give
		\[
		K_{\mathcal X_i}+\mathcal B_i\sim_{\Q,\mathcal Z_i}0.
		\]

	\end{proof}
	\subsection{Moduli of traditional stable minimal models}
	Here we use a descent-corrected version of the moduli of traditional stable minimal models studied in~\cite{jiang2025boundednesskltgoodminimal}.
	\begin{defn}[Traditional stable minimal models]\label{stmm}
		A \emph{traditional stable minimal model} \((X, B)\), \(A\) over an algebraically closed field \(\mtbb{K}\) of characteristic zero consists of a connected projective pair \((X, B)\) and a Cartier divisor \(A\) (not necessarily effective) such that
		\begin{itemize}
			\item \((X, B)\) is klt,
			\item \(K_X + B\) is semi-ample defining a contraction \(f : X \to Z\), and
			\item \(K_X + B + tA\) is ample for some \(t > 0\).
		\end{itemize}
		Let $d\in \N$, let $\Phi\subset \Q^{\geq0}$ and $\Gamma\subset \Q^{>0}$ be finite sets, and let $\sigma\in \Q[t]$ be a polynomial. A \emph{$(d, \Phi, \Gamma, \sigma)$-traditional stable minimal model} is a traditional stable minimal model \((X, B)\), \(A\) such that
		\begin{itemize}
			\item \(\dim X = d\),
			\item the coefficients of \(B\) are in \(\Phi\),
			\item \(\operatorname{vol}_F(A|_F) \in \Gamma\), where \(F\) is any general fiber of \(f : X \to Z\), and
			\item \((K_X + B + tA)^d = \sigma(t)\).
		\end{itemize}
	\end{defn}
	\begin{defn}
		Let \(S\) be a reduced scheme over \(\K\).
		For a projective flat morphism $X\to S$, let $\operatorname{Pic}_{X/S}$ denote the relative Picard functor, namely the fppf sheaf associated to
		\[
		T\longmapsto
		\operatorname{Pic}(X_T)/p_T^{\ast}\operatorname{Pic}(T),
		\qquad p_T:X_T\to T.
		\]
		\begin{enumerate}
			\item When \(S = \operatorname{Spec} \mathbb{K}_1\) for a field \(\mathbb{K}_1\), we define a \textbf{traditional stable minimal model} over \(\mathbb{K}_1\) as in Definition~\ref{stmm}, replacing \(\mathbb{K}\) by \(\mathbb{K}_1\) and connectedness by geometric connectedness.
			\item For general \(S\), a \textbf{family of traditional stable minimal models} over \(S\) consists of a projective morphism \(X \to S\), a \(\Q\)-divisor \(B\), and a class
			\[
			[A]\in\operatorname{Pic}_{X/S}(S)
			\]
			such that
			\begin{itemize}
				\item \((X, B) \to S\) is a locally stable family,
				\item for every geometric point $\bar s\to S$, the pair $(X_{\bar s},B_{\bar s})$ together with any line bundle representing $[A]_{\bar s}$ is a traditional stable minimal model over $k(\bar s)$.
			\end{itemize}
			We denote this family by \((X, B),[A] \to S\).
			\item We define the moduli functor $\mtcl M_{TS}(d, \Phi, \Gamma, \sigma)$ of \((d, \Phi, \Gamma, \sigma)\)-traditional stable minimal models from the category of reduced \(\K\)-schemes to the category of groupoids as follows:
			\begin{itemize}
				\item \textbf{On objects:} for a reduced \(\K\)-scheme \(S\), we take
				\begin{align*}
					&\mtcl M_{TS}(d, \Phi, \Gamma, \sigma)(S)\\
					=& \{\text{family of } (d, \Phi, \Gamma, \sigma)\text{-traditional stable minimal models over } S\}.
				\end{align*}

				An isomorphism
				\[
				((X',B'),[A'])\longrightarrow((X,B),[A])
				\]
				over $S$ is an isomorphism $\alpha_X:(X',B')\to(X,B)$ over $S$ such that
				\[
				[A']=\alpha_X^{\ast}[A]
				\quad\text{in }\operatorname{Pic}_{X'/S}(S).
				\]
				\item \textbf{On morphisms:} pullback along a morphism $\alpha:T\to S$ is given by divisorial pullback of $B$ and the pullback class $[A]_T\in\operatorname{Pic}_{X_T/T}(T)$. Morphisms in the resulting fibered category are the corresponding Cartesian isomorphisms preserving these data.
			\end{itemize}
		\end{enumerate}
	\end{defn}
	
	\begin{thm}[Corrected moduli theorem]\label{tsmoduli}
		The stack \(\mtcl M_{TS}(d, \Phi, \Gamma, \sigma)\) is a separated Deligne--Mumford stack of finite type and admits a coarse moduli space \(M_{TS}(d, \Phi, \Gamma, \sigma)\), which is a separated algebraic space of finite type.
	\end{thm}
	\begin{rem}
		The definition in \cite{jiang2025boundednesskltgoodminimal} uses a line bundle on $X$ considered up to tensoring by a line bundle pulled back from $S$. Such representatives do not in general satisfy fppf descent, whereas a section of the fppf sheaf $\operatorname{Pic}_{X/S}$ does. Thus we need such a correction to get a stack. The corrected theorem follows directly from the arguments in appendix A of \cite{jiang2025boundednesskltgoodminimal}. Indeed such arguments are quite standard and this issue in \cite{jiang2025boundednesskltgoodminimal} is better to be interpreted as a typo rather than a mathematical gap.
	\end{rem}
	
	\begin{lem}[Globalizing a relative Picard class]\label{globalpicardclass}
		Let $p:X\to S$ be a projective flat morphism with geometrically connected reduced fibers, where $S$ is of finite type, and let $[L]\in\operatorname{Pic}_{X/S}(S)$ be fiberwise ample. Then there exists a positive integer $q$ such that $[L]^{\otimes q}$ is represented by a line bundle on $X$.
	\end{lem}
	\begin{proof}
		Since $S$ has finitely many connected components, it suffices to prove the assertion for connected $S$ and then take a common multiple of the resulting exponents. Choose a finite fppf cover $S_i\to S$ with each $S_i$ of finite type and line bundles $L_i$ on $X_i:=X\times_S S_i$ representing $[L]|_{S_i}$. Write $p_i:X_i\to S_i$ for the projection. The $L_i$ are relatively ample. Relative Serre vanishing and cohomology and base change give a common integer $a>0$ such that
		\[
		E_i:=(p_i)_{\ast}L_i^{\otimes a}
		\]
		is locally free of positive rank and commutes with base change. The ranks agree on overlaps and hence descend to a locally constant function on $S$. Since $S$ is connected, they have a common value $r>0$.
		
		On $S_{ij}:=S_i\times_S S_j$, the representatives differ by a line bundle $M_{ij}$ from the base:
		\[
		L_j\simeq L_i\otimes p_{ij}^{\ast}M_{ij}.
		\]
		The projection formula gives
		\[
		E_j\simeq E_i\otimes M_{ij}^{\otimes a},
		\qquad
		\det E_j\simeq\det E_i\otimes M_{ij}^{\otimes ar}.
		\]
		Thus the base twists cancel in
		\[
		\widetilde L_i:=L_i^{\otimes ar}\otimes p_i^{\ast}(\det E_i)^{-1}.
		\]
		Since $p_{\ast}\O_X=\O_S$, any ambiguity in the overlap isomorphisms is a scalar from the base, which cancels between the two factors. The induced identifications of the $\widetilde L_i$ therefore satisfy the cocycle condition. Faithfully flat descent gives a line bundle $\widetilde L$ on $X$ representing $[L]^{\otimes ar}$.
	\end{proof}
	\subsection{Stack-theoretic preliminaries}
	\subsubsection{Compactification of tame Deligne--Mumford stacks}
	\begin{defn}
		Let \( f: X \to S \) be a morphism of Deligne--Mumford stacks and let \( \psi: I_{X/S} \to X \) be the inertia group scheme. We say that \( f \) is \textbf{tame} if for every point \( x: \operatorname{Spec}(k) \to X \), the group scheme \( \psi^{-1}(x) \to \operatorname{Spec}(k) \) has rank prime to the characteristic of \( k \). We say that \( f \) is \textbf{strictly tame} if for every point \( x: \operatorname{Spec}(k) \to X \) and specialization \( y: \operatorname{Spec}(k') \to S \) of \( f \circ x \), the group scheme \( \psi^{-1}(x) \to \operatorname{Spec}(k) \) has rank prime to the characteristic of \( k' \).
	\end{defn}
	
	Note that in characteristic zero, every morphism of Deligne--Mumford stacks is strictly tame. If \( S \) is equicharacteristic, then \( f: X \to S \) is tame if and only if \( f \) is strictly tame.
	\begin{thm}[Theorem F in~\cite{2011COMPACTIFICATIONOT}]\label{tame}
		Let $f : X \to S$ be a separated morphism of finite type between quasi-compact Deligne--Mumford stacks. The following conditions are equivalent:
		\begin{enumerate}
			\item $f$ is strictly tame.
			\item There exists a proper tame morphism $\bar{f} : \bar{X} \to S$ of Deligne--Mumford stacks and an open immersion $X \hookrightarrow \bar{X}$ over $S$.
		\end{enumerate}
		Moreover, $\bar{X}$ may be chosen so that, for every $x \in |\bar{X}|$, one has $\operatorname{stab}_{\bar{X}/S}(x) = A \times \prod_{i=1}^n \operatorname{stab}_{X/S}(x_i)$, where $A$ is an abelian group and $x_1, x_2, \ldots, x_n \in |X|$. Here $\operatorname{stab}(x)$ and $\operatorname{stab}(x_i)$ denote the geometric stabilizer groups.
	\end{thm}
	The reference~\cite{2011COMPACTIFICATIONOT} is an unpublished draft. For the particular moduli stacks of traditional stable minimal models used below, Proposition~\ref{tscompact} gives an independent proof of compactifiability from the quotient construction in~\cite{jiang2025boundednesskltgoodminimal}.
	\subsubsection{Bounding maps to Deligne--Mumford stacks via coarse moduli spaces}
	\begin{defn}
		A variety $U$ is \emph{smooth at infinity} if its singular locus is proper. A finitely presented morphism $U\to T$ between schemes is \emph{smooth at infinity} if the relative singular locus of $U/T$ is proper over $T$.
	\end{defn}
	\begin{defn}
		A \emph{compactifiable Deligne--Mumford stack} is a separated Deligne--Mumford stack of finite type over $\C$ that admits an open immersion into a proper Deligne--Mumford stack.
	\end{defn}
	\begin{defn}
		Let $\mathfrak{M}$ be a moduli functor, and let $V\to T$ be a bounded family of varieties parameterized by $T$. Define $\mathfrak{C}_{V/T,\mtfk{M}}$ to be the union, over all closed points $t\in T$, of the sets of morphisms $V_t\to\mtfk{M}$.
	\end{defn}
	\begin{defn}[boundedness]\label{bdddef}
		Let $\mtfk M$ be a moduli functor. Let $V\to T$ be as above. We say that a subset $\mtfk{C}\subset \mtfk{C}_{V/T,\mtfk M}$ is \emph{bounded} if there is a scheme $T^{\prime}\to T$ of finite type and a morphism $\Phi:V\times_T T^{\prime}=:V_{T^{\prime}}\to \mtfk{M}$, such that for every morphism $\mu:V_t\to \mtfk M$ arising from an element of $\mtfk C$, there is a closed point $t^{\prime}\in T^{\prime}$ such that $\mu$ coincides with $\Phi|_{V_{t^{\prime}}}$.
	\end{defn}
	\begin{defn}[Coarse boundedness]
		Let $\mtfk{M}$ be a moduli functor with coarse moduli space $M$ of finite type. Let $V\to T$ be as above. We say that a subset $\mtfk{C}\subset \mtfk{C}_{V/T,\mtfk M}$ is \emph{coarsely bounded} if there is a scheme $T^{\prime}\to T$ of finite type and a morphism $\Phi:V\times_T T^{\prime}=:V_{T^{\prime}}\to M$, such that for every morphism $\mu:V_t\to  M$ arising from an element of $\mtfk C$, there is a closed point $t^{\prime}\in T^{\prime}$ such that $\mu$ coincides with $\Phi|_{V_{t^{\prime}}}$.
	\end{defn}
	\begin{thm}[Theorem 1.7 in~\cite{kovacs2009boundednessfamiliescanonicallypolarized}, cf. Theorem 3.8 of~\cite{ascher2024boundednessresultsfamiliesnoncanonically}]\label{shaf}
		Let $\mtfk M$ be a compactifiable Deligne--Mumford stack. Let $U\to T$ be a morphism that is smooth at infinity, where $T$ is of finite type. Let $\mtfk{C}\subset \mtfk{C}_{U/T,\mtfk M}$ be a coarsely bounded subset. Then $\mtfk{C}$ is bounded in the sense of Definition~\ref{bdddef}.
	\end{thm}
	\begin{rem}
		The preceding statement is not a literal quotation of either~\cite[Theorem~1.7]{kovacs2009boundednessfamiliescanonicallypolarized} or~\cite[Theorem~3.8]{ascher2024boundednessresultsfamiliesnoncanonically}. The former assumes that the target stack is weakly bounded and concludes a uniform finite bound on the number of deformation types, while the latter treats morphisms from a fixed variety. Moreover, the notion of boundedness in Definition~\ref{bdddef} is slightly stronger than the usual one: the parameterizing scheme must be of finite type over $T$, and the universal source must be the base change $U\times_TT'$. All these modifications follow directly from the proof of~\cite[Theorem~1.7]{kovacs2009boundednessfamiliescanonicallypolarized}. Weak boundedness is used there only to construct a finite-type family of coarse moduli maps; once this is supplied by the coarse boundedness hypothesis, the remaining construction produces a finite-type parameter stack over $T$ whose universal morphism is defined on the corresponding base change of $U$. Taking a finite-type atlas of this stack gives precisely the boundedness asserted above.
	\end{rem}
	The following result is crucial for bounding the ``bad locus".
	\begin{thm}[Corollary 3.5 of~\cite{kovacs2009boundednessfamiliescanonicallypolarized}]\label{hartogslift}
		Let $U$ be a smooth variety, let $V\subset U$ be a big open subset, and let $\mathfrak M$ be a separated Deligne--Mumford stack with coarse moduli space $M$. If a morphism $\psi:U\to M$ admits a lift $\mu_V:V\to\mathfrak M$ over $V$, then $\mu_V$ extends uniquely, up to unique isomorphism, to a morphism
		\[
		\mu_U:U\longrightarrow\mathfrak M
		\]
		lifting $\psi$.
	\end{thm}
	
	\section{Proof of the Main Theorem}
	\subsection{Boundedness of rational maps}
	
	\begin{lem}[projective targets]\label{lem1}
		Let $M$ be a projective variety. Let $\mtcl P$ be a log bounded collection of couples, and fix a positive number $v$. Let $\mtcl F$ be a bounded collection of morphisms from projective curves to $M$, and let $\mtfk C$ be a collection of data
		\[
		\bigl(j:Z\dashrightarrow M,(Z,D),H\bigr),
		\]
		where $(Z,D)\in\mtcl P$ and $H$ is any very ample divisor on $Z$ satisfying $\operatorname{vol}_Z(H)\leq v$ and $\operatorname{deg}_HD\leq v$.
		Assume that, for a general  curve $C\subset Z$ cut out by $H^{\dim Z-1}$, the restriction $j|_C$ extends to a morphism
		\[
		j_C:C\longrightarrow M
		\]
		belonging to $\mtcl F$. Then there exist a flat family of polarized couples
		\[
		(\mtcl Z,\mtcl D;\mtcl H)\longrightarrow T
		\]
		and a rational map $\phi:\mtcl Z\dashrightarrow M$ with the following properties:
		\begin{itemize}
			\item every connected component of $T$ and of $\mtcl Z$ is a normal variety of finite type;
			\item $\mtcl H$ is a relatively very ample Cartier divisor on $\mtcl Z$;
			\item $\phi$ is defined on an open subset $U\subset\mtcl Z$ whose intersection with every fiber of $\mtcl Z\to T$ is dense; in particular, $\phi$ restricts to a rational map on every fiber;
			\item $\mtcl D$ contains no fiber of $\mtcl Z\to T$;
			\item for every datum $\bigl(j:Z\dashrightarrow M,(Z,D),H\bigr)$ in $\mtfk C$, there exist a closed point $t\in T$ and an isomorphism of couples
			\[
			(Z,D)\simeq(\mtcl Z_t,\mtcl D_t)
			\]
			under which $H\simeq\mtcl H_t$ and $j$ coincides with $\phi_t$.
		\end{itemize}
	\end{lem}
	\begin{proof}
		\medskip\noindent\emph{Step 1: bounding the graph.}
		By Lemma~\ref{finitecoeffrefinement}, we may choose finitely many parameterizing families for $\mtcl P$ in which the boundary and its coefficients are retained fiberwise. Choose an embedding $i:M\hookrightarrow \mtbb P^n$ and set $A:=i^{\ast}\O_{\mtbb P^n}(1)$. Since $H$ is very ample with bounded volume, all the polarized couples $(Z,D),H$ occurring in $\mtfk C$ form a bounded family. In particular, there is a positive number $\mu$, depending only on $\mtcl P$ and $v$, such that
		\[
		H^{\dim Z-1}\cdot D\leq \mu.
		\]
		Fix a positive integer $m$ such that for any $(Z,D),H$ occurring in $\mtfk C$, there exists an  embedding $\iota_Z:Z\hookrightarrow\mtbb P^m$ satisfying $H=\iota_Z^{\ast}\O_{\mtbb P^m}(1)$. Let $f:Z\dashrightarrow M$ be the underlying rational map of an element of $\mtfk C$. Let $\Gamma_f\subset\mtbb P^m\times\mtbb P^n$ denote the closure of the graph of $i\circ f$, and let $D_{\Gamma}$ denote the reduced closure of the graph of the restriction $\operatorname{Supp}(D)\dashrightarrow M$. Let $p$ and $q$ denote the projections from $\Gamma_f$ to $Z$ and $M$, respectively. We first prove the following claim.
		\begin{cla}
			There exists a positive number $\lambda$, depending only on $M,\mtcl P,\mtcl F,v$, such that
			\[
			\Gamma_f\cdot (p^{\ast}H+q^{\ast}A)^{\dim Z}\leq \lambda
			\]
			and
			\[
			D_{\Gamma}\cdot (p^{\ast}H+q^{\ast}A)^{\dim Z-1}\leq \lambda
			\]
			for the underlying map of every datum in $\mtfk C$.
		\end{cla}
		\begin{proof}[Proof of claim]
			After replacing $\mtcl P$ by finitely many equidimensional subfamilies, we may assume that every $(Z,D)\in\mtcl P$ is $d$-dimensional. Let $\widetilde\Gamma$ be the normalization of $\Gamma_f$, let $\widetilde p:\widetilde\Gamma\to Z$ and $\widetilde q:\widetilde\Gamma\to M$ be the projections, and set
			\[
			P:=\widetilde p^{\ast}H,
			\qquad
			L:=\widetilde q^{\ast}A.
			\]
			Both $P$ and $L$ are nef, and $P^d=H^d\leq v$. If $C\subset Z$ is a general  curve cut out by $H^{d-1}$, then
			\[
			P^{d-1}\cdot L=\deg_C j_C^{\ast}A.
			\]
			The right-hand side is uniformly bounded because the morphisms $j_C$ belong to the bounded family $\mtcl F$.
			
			Choose a general divisor $G\in|L|$ and set $G_Z:=\widetilde p_{\ast}G$. The divisor $G$ has no $\widetilde p$-exceptional component, and
			\[
			H^{d-1}\cdot G_Z=P^{d-1}\cdot L
			\]
			Thus the effective divisors $G_Z$ form a bounded family. Thus there is a positive integer $a$, depending only on $M,\mtcl P,\mtcl F,v$, and an effective divisor $E_f\in|aH|$ containing $G_Z$. The pullback $\widetilde p^{\ast}E_f$ contains $G$, so $aP-L$ is pseudo-effective. Therefore, for every $0\leq k\leq d-1$,
			\[
			P^k\cdot(P+L)^{d-k}
			\leq(a+1)P^{k+1}\cdot(P+L)^{d-k-1}.
			\]
			Iterating these inequalities gives
			\[
			(P+L)^d\leq(a+1)^d P^d\leq(a+1)^d v.
			\]
			Thus the degrees of the graphs $\Gamma_f$ are uniformly bounded.
			
			Finally, boundedness of the polarized couples $(Z,D),H$ gives a positive integer $b$, depending only on $M,\mtcl P,\mtcl F,v$, and an effective divisor $E_D\in|bH|$ containing $\operatorname{Supp}(D)$. If $D_{\widetilde\Gamma}$ denotes the  strict transform of $D$ on $\widetilde\Gamma$, then $\widetilde p^{\ast}E_D$ contains $D_{\widetilde\Gamma}$, so $bP-D_{\widetilde\Gamma}$ is pseudo-effective. Hence
			\[
			D_{\widetilde\Gamma}\cdot(P+L)^{d-1}
			\leq bP\cdot(P+L)^{d-1},
			\]
			and the right-hand side is at most $b(a+1)^{d-1}P^d\leq b(a+1)^{d-1}v$. The required bound for $D_\Gamma$ follows.
		\end{proof}
		\medskip\noindent\emph{Step 2: constructing the parameterizing family.}
		The claim shows that the data $D_\Gamma\subset\Gamma_f\subset Z\times M\subset\mtbb P^m\times\mtbb P^n$ form a bounded family. Thus we may choose a reduced finite-type scheme $T$, a flat family of polarized couples
		\[
		(\mtcl Z,\mtcl D;\mtcl H)\longrightarrow T,
		\]
		and a projective flat $T$-scheme $\mtcl G$ equipped with a closed immersion
		\[
		\mtcl G\hookrightarrow\mtcl Z\times M
		\]
		such that $\mtcl H$ is a relatively very ample Cartier divisor and the following holds: for every datum $\bigl(j:Z\dashrightarrow M,(Z,D),H\bigr)$ in $\mtfk C$, there is a closed point $t\in T$ and an isomorphism of polarized couples
		\[
		(Z,D;H)\simeq(\mtcl Z_t,\mtcl D_t;\mtcl H_t)
		\]
		under which $\mtcl G_t\subset\mtcl Z_t\times M$ can be identified with the closure $\Gamma_f\subset Z\times M$ of the corresponding graph. Let $\rho:\mtcl G\to\mtcl Z$ and $\sigma:\mtcl G\to M$ denote the induced projections.
		
		Using generic flatness and  Theorem 12.1.6 in~\cite{PMIHES_1967__32__5_0}, by possibly passing to a finite stratification on $T$ and discarding irrelevant strata,  we may assume that $T$ is normal, $\mtcl Z$ and $\mtcl D$ are flat over $T$, the fibers of $\mtcl Z$ are geometrically normal, $\mtcl D$ contains no fiber, and $\rho$ is an isomorphism over an open subset $U\subset\mtcl Z$ meeting every fiber densely. 
		Since $\mtcl Z\to T$ is flat with geometrically normal fibers, Proposition 6.8.3 in~\cite{PMIHES_1967__32__5_0} implies that $\mtcl Z$ is normal. As $\rho$ is an isomorphism over $U$, the morphism
		\[
		\phi:=\sigma\circ\rho^{-1}:U\longrightarrow M
		\]
		defines a rational map on $\mtcl Z$. 
	\end{proof}
	\begin{lem}[Algebraic-space targets]\label{algspacetarget}
		Let $M$ be a reduced separated algebraic space of finite type, let $\mtcl P$ be a log bounded collection of couples, and fix a positive number $v$. Let $\mtcl F$ be a bounded collection, in the sense of Definition~\ref{bdddef}, of morphisms from quasi-projective curves to $M$. Suppose that $\mtfk C$ is a collection of data
		\[
		\bigl(j:(Z,D)\dashrightarrow M,H\bigr),
		\]
		where $(Z,D)\in\mtcl P$ and $H$ is any very ample divisor on $Z$ satisfying $\operatorname{vol}_Z(H)\leq v$ and $\deg_HD\leq v$. Assume that, for a general  curve $C\subset Z$ cut out by $H^{\dim Z-1}$, there is a nonempty open subset $C^{\circ}\subset C$ such that the restriction
		\[
		j_C:C^{\circ}\longrightarrow M
		\]
		is a morphism belonging to $\mtcl F$. Then there exist a flat family of polarized couples
		\[
		(\mtcl Z,\mtcl D;\mtcl H)\longrightarrow T
		\]
		and a rational map $\phi:\mtcl Z\dashrightarrow M$ with the following properties:
		\begin{itemize}
			\item every connected component of $T$ and of $\mtcl Z$ is a normal variety of finite type;
			\item $\mtcl H$ is a relatively very ample Cartier divisor on $\mtcl Z$;
			\item $\phi$ is defined on an open subset $V\subset\mtcl Z$ whose intersection with every fiber of $\mtcl Z\to T$ is dense;
			\item $\mtcl D$ contains no fiber of $\mtcl Z\to T$;
			\item for every datum $\bigl(j:(Z,D)\dashrightarrow M,H\bigr)$ in $\mtfk C$, there exist a closed point $t\in T$ and an isomorphism of couples
			\[
			(Z,D)\simeq(\mtcl Z_t,\mtcl D_t)
			\]
			under which $H\simeq\mtcl H_t$ and $j$ coincides with $\phi_t$.
		\end{itemize}
	\end{lem}
	\begin{proof}
		\medskip\noindent\emph{Step 1: passing to stratification on the target.}
		A finite-type separated algebraic space has a dense open subspace which is a scheme; applying  this repeatedly gives a finite disjoint stratification
		\[
		M=\coprod_{i=1}^s M_i
		\]
		by quasi-projective locally closed subschemes; see, for example,~\cite[Tag 06LY]{stacks-project}. Fix an open immersion $M_i\hookrightarrow\overline M_i$ into a projective scheme for every $i$.

		Choose a family of quasi-projective curves $\mtcl C^{\circ}\to S$ over finite type base $S$ and a morphism $\psi:\mtcl C^{\circ}\to M$ realizing $\mtcl F$. By compactifying and passing to a finite stratification on $S$, we may assume that $\mtcl C^{\circ}\to S$ is smooth and there is a smooth projective family of curves $\bar{\mtcl C}\to S$ containing $\mtcl C^{\circ}$ as an open subfamily.
		
		Stratify $S$ further so that  over each connected component $S_{\lambda}$, the  generic image of every fiber lies in a fixed $M_i$. Let $\mtcl C^{\circ}_{\lambda}\subset\overline{\mtcl C_{\lambda}}\to S_{\lambda}$ denote the restriction, then $\psi$ induces a rational map
		\[
		\overline{\mtcl C_{\lambda}}\dashrightarrow M_i\hookrightarrow\overline M_i.
		\]
		On every fiber this map extends uniquely to a morphism, because $\overline M_i$ is projective. Using boundedness of graphs of these maps, and possibly changing $S_{\lambda}$, we may assume $\overline{\mtcl C_{\lambda}}\dasharrow\overline{M_i}$ extends to a morphism on $\overline{\mtcl C_{\lambda}}$.
		
		\medskip\noindent\emph{Step 2: applying Lemma \ref{lem1}.}
		For each $i$, let $\mtfk C_i\subset\mtfk C$ be the subcollection consisting of data $(j:(Z,D)\dasharrow M,H)$ such that the generic image of $j$ lies in $M_i$. Then we have $\mtfk C=\bigcup_i\mtfk C_i$. Fix one of $\mtfk C_i$ and take $(j:(Z,D)\dasharrow M,H)\in \mtfk C_i$. Let $C$ denote a general curve cut by $H^{\dim Z-1}$ and let $j_{C,i}:C\dasharrow M_i$ denote the restriction of $j$ on $C$. 
		Step~1 shows that $j_{C,i}$ extends to a morphism $C\to\overline M_i$ belonging to a bounded family. Lemma~\ref{lem1}, applied to the projective target $\overline M_i$, therefore bounds the rational maps. 
		Taking the finite disjoint union over $i$ gives the required boundedness assertion.
	\end{proof}
	\subsection{Moduli maps over curves}
	\begin{lem}\label{curvemap}
		Let $d\in \N$, let $v,v_F,r,\varepsilon\in\Q^{>0}$ with $v_F\leq v$, and let $\Phi\subset[0,1]\cap \Q$ be a finite set. There exist a positive integer $l$ and a  log bounded family $\mtcl P$ of couples, depending only on $(d,\Phi,v,r,\varepsilon)$, with the following property. Let $f:(X,B),A\to (C,H)$ be a weak $(d,\Phi,v,r,\varepsilon)$-polarized log Calabi--Yau fibration over a curve $C$, and assume that $\operatorname{vol}_F(A|_F)=v_F$ for a general fiber $F$. Then there is a  reduced divisor $D$ on $C$ such that $(C,D)\in\mtcl P$ and, writing $U:=C\setminus D$, the restriction $f_U:(X_U,B_U),[lA_U]\to U$ is a family of traditional stable minimal models. Moreover, as $f$ varies, the corresponding moduli maps
		\[
		\mu_f:U\longrightarrow \mtcl M_{TS}:=\mtcl {M}_{TS}\bigl(d-1,\Phi,\{l^{d-1} v_F\},l^{d-1} v_F t^{d-1}\bigr)
		\]
		form a bounded family in the sense of Definition~\ref{bdddef}. The same conclusions remain valid after replacing $l$ by any fixed positive multiple.
	\end{lem}
	\begin{proof}
		By Theorem~\ref{curvecase}, there are a positive integer $l$ and a constant $\lambda$ such that every fibration under consideration admits $A'\sim_C A$ with
		\[
		L:=lA'
		\]
		very ample and $L^d\leq l^d\lambda$. The polarized varieties $(X,L)$ are therefore bounded. Moreover,
		\[
		L^{d-1}\cdot f^{\ast}H=l^{d-1}v_F\deg H
		\]
		is bounded. Since the coefficients of $B$ lie in the fixed finite set $\Phi$ and $K_X+B\sim_{\Q}f^{\ast}N$ with $\deg N\leq\deg H\leq r$, the numbers $L^{d-1}\cdot B_{\mathrm{red}}$ are bounded as well. Hence possibly replacing $l$ with a bounded multiple, we may assume that
		\[
		L-B_{\mathrm{red}}
		\qquad\text{and}\qquad
		L-f^{\ast}H
		\]
		are pseudo-effective.

		Applying Lemma~\ref{hanjiang} to $(X,B),L\to(C,H)$ gives finitely many diagrams
		\[
		\begin{tikzcd}
			{(\mathcal X_i,\mathcal B_i),\mathcal J_i}
			\arrow[r,"g_i"]
			\arrow[dr]
			&
			{\mathcal C_i}
			\arrow[d]
			\\
			&
			T_i
		\end{tikzcd}
		\]
		over bases of finite type such that $K_{\mtcl X_i}+\mtcl B_i\sim_{\Q,\mtcl C_i}0$ and every
		\[
		f:(X,B),lA^{\prime}\longrightarrow C
		\]
		is isomorphic to a fiber of one of these diagrams, including its boundary, fibration, and polarization. 
		
		After a further stratification, the numerical invariants are constant and every $\mathcal J_i$ is relatively ample. By Lemma~\ref{genericls}, possibly passing to a further stratification, we may assume there is an open subset $\mathcal U_i\subset\mathcal C_i$, meeting every relevant fiber, such that
		\[
		(\mathcal X_i,\mathcal B_i)|_{\mathcal U_i},
		[\mathcal J_i|_{\mathcal U_i}]
		\longrightarrow\mathcal U_i
		\]
		is a family of traditional stable minimal models. After further stratification, we may assume that $\mathcal U_i\to T_i$ is smooth and   $\mathcal D_i:=\operatorname{red}(\mathcal C_i\setminus\mathcal U_i)$ is flat with reduced divisorial fibers over $T$.  The couples
		\[
		\bigl((\mathcal C_i)_t,\operatorname{red}((\mathcal D_i)_t)\bigr)
		\]
		form a log bounded family. The stable-minimal-model family induces a morphism
		\[
		\mathcal U_i\longrightarrow
		\mtcl M_{TS}\bigl(d-1,\Phi,\{l^{d-1} v_F\},
		l^{d-1} v_F t^{d-1}\bigr).
		\]
		The finitely many morphisms above form a bounded collection in the sense of Definition~\ref{bdddef}.
		
		Finally, $A'\sim_C A$, so $lA'$ and $lA$ determine the same class in the relative Picard functor. Hence the moduli morphism associated with the original polarization is a fiber of one of the finitely many families above. This proves the lemma.
	\end{proof}
	
	\subsection{Controlling bad locus}
	\begin{lem}\label{stablecomparison}
		Let $S$ be a smooth variety and let $V\subset S$ be an open subset with $\operatorname{codim}_S(S\setminus V)\geq2$. Let $f:(X,B)\to S$ be  log Calabi-Yau fibration from a klt pair. Let $A$ be an $f$-ample integral $\Q$-Cartier divisor on $X$ whose restriction to $X_V:=f^{-1}(V)$ is Cartier. Let $g:(Y,\Delta)\to S$ be a projective locally stable family with  klt fibers, and let $L$ be a $g$-ample line bundle on $Y$. Suppose that there is an isomorphism
		\[
		\alpha_V:(X_V,B_V)\xrightarrow{\sim}(Y_V,\Delta_V)
		\]
		over $V$ such that $A|_{X_V}\sim_{V}\alpha_V^{\ast}L_{Y_V}$. Then $\alpha_V$ extends uniquely to an isomorphism $\alpha:(X,B)\xrightarrow{\sim}(Y,\Delta)$ over $S$. Moreover, $A$ is Cartier and
		\[
		A\sim_S \alpha^{\ast}L
		\]
	\end{lem}
	\begin{proof}
		By Lemma~\ref{fibtot}, $Y$ is normal and $(Y,\Delta)$ is klt. Flatness of $g$ implies that $Y\setminus Y_V$ has codimension at least two. Thus $(Y|_V,\Delta|_V)$ is log Calabi-Yau over $V$ implies $(Y,\Delta)$ is log Calabi-Yau over $S$. In particular, $(X,B)$ and $(Y,\Delta)$ are crepant birational. 
		
		By \cite[Corollary~4.55]{book}, $f:(X,B)\to S$ is locally stable. Thus $X\setminus X_V$ also has codimension at least two. Thus the birational map $Y\dasharrow X$ induced by $\alpha_V$ is small and it is the log canonical model of $(Y,\Delta+tL)$, where $t>0$ is a sufficiently small rational number. Thus $\alpha_V$ extends to an isomorphism $\alpha:X\to Y$ and other assertions follows.

	\end{proof}
	\begin{lem}\label{bddopen}
		Fix $d\in\N$, $v,r,\varepsilon\in\Q^{>0}$, and a finite set $\Phi\subset[0,1]\cap\Q$. There exist a positive integer $l$ and a log bounded family $\mtcl P$ of couples, depending only on these data, with the following property. For every weak $(d,\Phi,v,r,\varepsilon)$-polarized log Calabi--Yau fibration
		\[
		f:(X,B),A\longrightarrow(Z,H),
		\]
		there is a reduced divisor $D$ on $Z$ and an big open subset $V\subset U:=Z\setminus D$ such that,
		\begin{itemize}
			\item $(Z,D)\in\mtcl P$ and $U$ is smooth, and
			\item $(X_V,B_V),[lA_V]\to V$ is a family of traditional stable minimal models.
		\end{itemize}
	\end{lem}
	\begin{proof}
		Let $I_f$ denote the collection of nonempty smooth open subsets of $Z$ over which $f$ is locally stable with klt fibers.
		By Lemma \ref{genericls}, $I_f$ is non-empty. Since $I_f$ is closed under union, there exists a maximal element in $I_f$. Let $W$ denote the maximal element in $I_f$. 
		By~\cite[Corollary~1.6]{birkar2022geometrypolarisedvarieties}  and  \cite[Theorem~1.2]{Han_2026}, there exists a uniform positive integer $l$ such that  $lA$ restricts to a relative Cartier divisor over $W$. Thus $f:(X,B),lA\to Z$ restricts to a family of traditional stable minimal models over $W$. Let $E$ be the reduced divisorial part of $Z\setminus W$.
		
		Put $m:=\dim Z$. If $m=0$, there is nothing to prove, so assume $m>0$. Let $C\subset Z$ be a general curve cut out by $H^{m-1}$. It meets every component of $E$ transversally at a general point and avoids the codimension at least two part of $Z\setminus W$. By adjunction, the restricted fibration over $C$ is a weak
		\(
		(d-m+1,\Phi,v,dr,\varepsilon)
		\)
		-polarized log Calabi--Yau fibration. By Lemma~\ref{curvemap} there exists a reduced divisor $D_C$ of uniformly bounded degree such that the restriction over $C\setminus D_C$ is a family of traditional stable minimal models.
		
		Since $C\subset Z_{\mathrm{reg}}$, by Theorem 4.54 in \cite{book} and adjunction, the locally stable locus of the restriction $f_C$ equals the locally stable locus of $f$ restricted to $C$. Consequently
		\[
		C\cap E\subseteq D_C,
		\qquad
		H^{m-1}\cdot E\leq\deg D_C.
		\]
		The right-hand side is uniformly bounded because the couples $(C,D_C)$ form a log bounded family. Thus the couples $(Z,E)$ are log bounded.
		
		By boundedness, one can choose a uniform $a>0$ such that for each $(Z,H)$ there exists an effective divisor $Q\in|aH|$ containing $Z_{\mathrm{sing}}$. One can also choose $Q$ such that it has no component in common with $E$. Set
		\[
		D:=E+Q,
		\qquad U:=Z\setminus D,
		\qquad V:=W\cap U.
		\]
		Then $U$ is smooth, $U\setminus V$ has codimension at least two, and $(Z,D)$ forms a log bounded family. By construction, $f_V:(X_V,B_V),[lA_V]\to V$ is a family of traditional stable minimal models.
	\end{proof}
	
	\subsection{Bounded models and conclusion}
	\begin{thm}\label{corethm}
		Let \(d \in \mathbb{N}\), let \(v,r,\varepsilon \in \mathbb{Q}^{>0}\), and let \(\Phi \subset [0,1]\cap\mathbb{Q}\) be a finite set. Then there exist a positive integer \(l\) and a log bounded family \(\mathcal{P}\) of pairs, depending only on \(d,\Phi,v,r,\varepsilon\), with the following property.
		Assume that \(f:(X,B),A\to(Z,H)\) is a weak \((d,\Phi,v,r,\varepsilon)\)-polarized log Calabi--Yau fibration and that \(G\geq0\) is a general element of \(|6dH|\). Then there exist a normal variety \(V\), a birational map $V\dasharrow X$, and an effective integral divisor \(J\) on \(V\) such that
		
		\begin{itemize}
			\item there is a contraction \( h : V \to Z \) and \( V \) is \( \mathbb{Q} \)-factorial,
			\item the birational map \(V\dashrightarrow X\) over \(Z\) is an isomorphism in codimension one,
			\item writing $B_V$ for the strict transform of $B$ on $V$, one has
			\[
			\left(V,B_V+\frac12h^{\ast}G+J\right)\in\mathcal P,
			\]
			\item \(J|_{V_\eta}\sim lA_V|_{V_\eta}\), where \(\eta\) is the generic point of \(Z\) and \(A_V\) is the strict transform of \(A\) on \(V\).
		\end{itemize}
	\end{thm}
	\begin{proof}
		Set $m:=\dim Z$. The case $m=0$ follows directly from~\cite[Corollaries~1.4 and~1.6]{birkar2022geometrypolarisedvarieties}. Since $m\leq d$ has only finitely many positive possibilities when $d$ is fixed, we fix $m\geq1$ for the rest of the proof. Put $n:=d-m$, let $F$ be a general fiber of $f$, and set
		\[
		v_F:=\operatorname{vol}_F(A|_F).
		\]
		By~\cite[Corollary~1.6]{birkar2022geometrypolarisedvarieties}, and  \cite[Theorem~1.2]{Han_2026}, $v_F$ has only finitely many possible values, so we fix one of them as well.
		
		Lemma~\ref{bddopen} gives a log bounded family $\mtcl P$ of couples, an integer $l>0$, a reduced divisor $D$ on $Z$, such that $(Z,D)\in \mtcl P$, $U:=Z\setminus D$ is smooth and for some big open subset $V\subset U$
		\[
		(X_V,B_V),[lA_V]\longrightarrow V
		\]
		is a family of traditional stable minimal models.
		Since $G\in|6dH|$ is general, we may assume that it is reduced and has no component in common with $D$. 
		
		Put
		\[
		D^{\dagger}:=D+\frac12G.
		\]
		Then the couples $(Z,D^{\dagger})$ form a log bounded family whose coefficients belong to $\{1/2,1\}$. 
		
		\medskip\noindent\emph{Step 1: bound the coarse moduli map and remove its indeterminacy.}
		Let $C\subset Z$ be a general curve cut out by $H^{m-1}$, and denote the restricted fibration by
		\[
		f_C:(X_C,B_{X_C}),A_C\longrightarrow C.
		\]
		Then $\dim X_C=n+1$. Choose $N$ on $Z$ with $K_X+B\sim_{\Q}f^{\ast}N$, write $H_C:=H|_C$, and set
		\[
		N_C:=N|_C+(m-1)H_C,\qquad \widehat H_C:=mH_C.
		\]
		Adjunction gives
		\[
		K_{X_C}+B_{X_C}\sim_{\Q}f_C^{\ast}N_C,\qquad
		\widehat H_C-N_C=(H-N)|_C,
		\]
		and
		\[
		\deg_C\widehat H_C=mH^m\leq dr.
		\]
		Moreover, $(X_C,B_{X_C})$ is $\varepsilon$-lc and its coefficients belong to $\Phi$. Hence $f_C$ is a weak $(n+1,\Phi,v,dr,\varepsilon)$-polarized log Calabi--Yau fibration, and its general fibers are of volume $v_F$.
		
		Set
		\[
		\Gamma:=\{l^{n}v_F\},\qquad \sigma(s):=l^{n}v_F s^{n},
		\]
		and write
		\[
		\mathfrak M:=\mtcl M_{TS}(n,\Phi,\Gamma,\sigma),\qquad
		M:=M_{TS}(n,\Phi,\Gamma,\sigma).
		\]
		After replacing $l$ by a multiple if necessary, Lemma~\ref{curvemap} gives a nonzero reduced divisor $D_C$ on $C$ and a morphism $C^{\circ}:=C\setminus D_C\to\mathfrak M$; the couples $(C,D_C)$ are log bounded and these morphisms form a bounded collection. Theorem~\ref{tsmoduli} says that $\mathfrak M$ is a separated Deligne--Mumford stack of finite type with coarse moduli space $M$, a separated algebraic space of finite type. The stack $\mathfrak M$ is compactifiable by Proposition~\ref{tscompact}; alternatively, one may use Theorem~\ref{tame}. The family over $V$ defines a rational coarse moduli map
		\[
		j:Z\dashrightarrow M.
		\]
		By Lemma~\ref{curvemap}, the restrictions $j_C:C\dasharrow M$  form a bounded collection.
		
		By Lemma~\ref{algspacetarget} and Lemma \ref{finitecoeffrefinement}, there exists a flat family with divisorial data
		\[
		(\mtcl Z_0,\mtcl D_0^{\dagger};\mtcl H_0)\longrightarrow T_0
		\]
		and a rational map $\phi_0:\mtcl Z_0\dashrightarrow M$ defined on an open subset $V_0\subset\mtcl Z_0$ meeting every fiber densely, such that for every $j:Z,D^{\dagger};H)\dasharrow M$ as above, there exists $t\in T_0$ with $(Z,D^{\dagger};H)\simeq ((\mtcl Z_0)_t,(\mtcl D_0^{\dagger})_t;(\mtcl H_0)_t)$ and $j=(\phi_0)_t$ under this identification. Write
		\[
		\mtcl D_0^{\dagger}=\mtcl E_0+\frac12\mtcl G_0,
		\]
		where $\mtcl E_0$ and $\mtcl G_0$ are the sums of the coefficient-$1$ and coefficient-$1/2$ components, respectively.

		Then every relevant fiber of $\mtcl G_0$ is the prescribed divisor $G$, which is an ample Cartier divisor.  
		Thus by the fiberwise Cartier criterion~\cite[Tag~062Y]{stacks-project}, after possibly shrinking $T_0$, we may assume $\mtcl G_0$ is a relative effective Cartier divisor which is ample over $T_0$.
		
		Set
		\[
		\mtcl U_0:=\mtcl Z_0\setminus\operatorname{Supp}(\mtcl E_0).
		\]
		By possibly enlarging $\mtcl E_0$, we may assume $\mtcl U_0\to T_0$ is smooth, and $\Phi_0:=\phi_0|_{\mtcl U_0}:\mtcl U_0\to M$ is a morphism. By possibly passing to a finite stratification on $T_0$, we may assume that $\mtcl U_0$ is dense in every fiber. For $f:(X,B),A\to (Z,H)$ as above, let $U\subset Z$ be a corresponding fiber of $\mtcl U_0$. Since $j|_U:U\to M$ is a morphism, by Lemma \ref{bddopen} and Theorem \ref{hartogslift},  $j|_U$ can be uniquely lifted to a morphism $U\to \mtfk M$.

		\medskip\noindent\emph{Step 2: lift the moduli maps and spread out the polarized families.}
		By Step~1 and Theorem~\ref{shaf}, there exists a finite-type morphism $\mtcl T\to T_0$ such that, setting
		\[
		\mtcl U:=\mtcl U_0\times_{T_0}\mtcl T\longrightarrow\mtcl T,
		\]
		there exists a morphism $\Psi:\mtcl U\to\mathfrak M$ with the following property: every extension $\mu_U$ constructed in Step~1 is obtained by restricting $\Psi$ to a fiber.
		
		By possibly passing to a finite stratification on the family induced by $\Psi$, we obtain a reduced normal scheme $T$ of finite type and a family
		\[
		\phi_U:(\mtcl{X_U},\mtcl {B_U}),[\mathscr A]\longrightarrow
		\mtcl U\subset\mtcl Z\longrightarrow T.
		\]
		By pulling back $\mtcl G_0$, we get a relative effective Cartier divisor $\mtcl G_{\mtcl Z}$ on $\mtcl Z$ which is ample over $T$. By applying  Lemma \ref{globalpicardclass} to $\phi_U$ and possibly changing $l$ and $[\mathscr A]$, we may assume there exists a line bundle $\mtcl {J_U}$ on $\mtcl {X_U}$ representing $[\mathscr A]$. By possibly passing to a stratification on $T$ and possibly shrinking $\mtcl U$, we may assume $\mtcl {J_U}$ is a relative effective Cartier divisor over $\mtcl U$. 
		
		By Lemma \ref{stablecomparison}, for every $f:(X,B),A\to (Z,H)$ as above, there is a closed point $t\in T$ and an identification $\mtcl U_t\subset\mtcl Z_t\simeq U\subset Z$ under which $(\mtcl G_{\mtcl Z})_t=G$ and
		\[
		\bigl((\mtcl {X_U})_t,(\mtcl {B_U})_t\bigr),[\mathscr A]_t
		\longrightarrow\mtcl U_t
		\]
		restricts over $U$ to $(X_U,B_U),[lA_U]\to U$.

		\medskip\noindent\emph{Step 3: extend and resolve the family.}
		Choose a projective extension
		\[
		\phi:\mtcl X\longrightarrow\mtcl Z
		\]
		of $\phi_U$, so that $\mtcl X\times_{\mtcl Z}\mtcl U\simeq\mtcl {X_U}$, and replace $\mtcl X$ by its normalization. Let $\mtcl B$ and $\mtcl J$ be the closures of $\mtcl {B_U}$ and $\mtcl {J_U}$ on $\mtcl X$, and set
		\[
		\mtcl G_{\mtcl X}:=\phi^{\ast}\mtcl G_{\mtcl Z}.
		\]
		
		By possibly passing to a stratification on $T$, we may assume $\mtcl X\to T$ is flat and $\mtcl{G_X}$ is a relative effective Cartier divisor.
		
		Take a resolution
		\[
		r:\mtcl Y^{\prime}\longrightarrow\mtcl X
		\]
		
		such that
		\[
		\left(\mtcl Y^{\prime},
		\mtcl B^{\prime}+\mtcl G^{\prime}+\mtcl J^{\prime}
		+\mtcl E^{\prime}+\mtcl F^{\prime}\right)
		\]
		is relatively log smooth over $T$. Here $\mtcl B^{\prime},\mtcl G^{\prime},\mtcl J^{\prime}$ are the strict transforms of $\mtcl B,\mtcl G_{\mtcl X},\mtcl J$, respectively, $\mtcl E^{\prime}$ is the closure of the reduced exceptional divisor of
		\[
		r^{-1}(\mtcl X_U)\longrightarrow\mtcl X_U,
		\]
		and
		\[
		\mtcl F^{\prime}:=
		\left(\mtcl Y^{\prime}\setminus r^{-1}(\mtcl X_U)\right)_{\mathrm{red}}.
		\]
		Define
		\[
		\Delta^{\prime}:=\mtcl B^{\prime}+\frac12\mtcl G^{\prime}
		+\left(1-\frac{\varepsilon}{2}\right)(\mtcl E^{\prime}+\mtcl F^{\prime}).
		\]
		The pair $(\mtcl Y^{\prime},\Delta^{\prime})$ is an $\varepsilon/2$-lc pair which is log smooth over $T$. By successively blowing up, we may further assume that $(\mtcl Y^{\prime},\Delta^{\prime})$ has terminal fibers over $T$. 
		
		\medskip\noindent\emph{Step 4: run the MMP in family.}
		Fix $f:(X,B),A\to (Z,H)$ and $t\in T$ correspondingly. The identification over  $\mtcl U_t\simeq U$ induces a birational map $\pi:\mtcl Y_t^{\prime}\dashrightarrow X$. If $p:W\to\mtcl Y_t^{\prime}$ and $q:W\to X$ resolve its graph, then the $\varepsilon$-lc property of $(X,B)$ and the negativity lemma give
		\[
		p^{\ast}(K_{\mtcl Y_t^{\prime}}+\Delta_t^{\prime})
		\geq q^{\ast}(K_X+B).
		\]
		Since $(\mtcl Y_t^{\prime},\Delta_t^{\prime})$ is terminal, the displayed inequality implies that $\pi$ is a birational contraction.
		
		Take a small $\mathbb Q$-factorialization $\rho:X^{\prime}\to X$, set $f^{\prime}:=f\circ\rho$, and let $B^{\prime}_X$ be the strict transform of $B$. Then $(X^{\prime},B^{\prime}_X)$ is $\mathbb Q$-factorial and $\varepsilon$-lc, and the induced map $\pi^{\prime}:\mtcl Y_t^{\prime}\dashrightarrow X^{\prime}$ is a birational contraction. Let $E_t^{\prime}$ be its reduced exceptional divisor and let $\Gamma_t^{\prime}$ be the sum of $(1-\varepsilon/2)E_t^{\prime}$ and the strict transform of
		\[
		B_X^{\prime}+\frac12(f^{\prime})^{\ast}G.
		\]
		The canonical bundle formula writes $N=K_Z+B_Z+\mathbf M_Z$ with $(Z,B_Z,\mathbf M)$ generalized lc. The generalized cone theorem~\cite[Theorem~5.1(2)]{haconliu2023flips} shows that $N+2dH$ is nef, so $N+3dH$ is ample. Hence
		\[
		K_{X^{\prime}}+B_X^{\prime}+\frac12(f^{\prime})^{\ast}G
		\sim_{\Q}(f^{\prime})^{\ast}(N+3dH)
		\]
		is semi-ample and defines $f^{\prime}$. Consequently
		\[
		(\mtcl Y_t^{\prime},\Gamma_t^{\prime})\dashrightarrow
		\left(X^{\prime},B_X^{\prime}+\frac12(f^{\prime})^{\ast}G\right)
		\]
		is a good minimal model.
		
		By construction, the support of $\Gamma_t^{\prime}$ is contained in the $\operatorname{Supp}\Delta^{\prime}_t$, and its coefficients belong to the finite set
		\[
		\Phi\cup\left\{\frac12,1-\frac{\varepsilon}{2}\right\}.
		\]
		By Lemma \ref{finitecoeffrefinement}, after possibly changing $T$, we may assume there exists a divisor $\Gamma^{\prime}$ on $\mtcl Y^{\prime}$ with
		\[
		\Gamma_t^{\prime}=\Gamma^{\prime}|_{\mtcl Y_t^{\prime}}
		\]
		for every relevant point $t\in T$.
		By possibly  passing to a stratification on $T$ and discarding irrelevant components, we may assume $T$ is smooth and  each connected component of $T$ contains a point whose fiber has a good minimal model. By Theorem~1.2 of~\cite{hacon2017boundednessmodulivarietiesgeneral}, possibly further stratify, we may assume $(\mtcl Y^{\prime},\Gamma^{\prime})$ admits a fiberwise good minimal model $(\mtcl V,\Gamma)\to T$. By boundedness of extremal rays, the induced map $\mtcl V\dasharrow \mtcl Z$ is a morphism.
		
		Let $\mtcl{J_V}$ be the strict transform of $\mtcl J^{\prime}$ on $\mtcl V$. For every relevant $t$, Lemma~2.4 of~\cite{hacon2012existencelogcanonicalclosures} shows that $(\mtcl V_t,\Gamma_t)$ is isomorphic in codimension one to
		\[
		\left(X^{\prime},B_X^{\prime}+\frac12(f^{\prime})^{\ast}G\right),
		\]
		and hence to $(X,B+\frac12f^{\ast}G)$. By Lemma~\ref{stricttransformlinearequiv}, over  $U\simeq \mtcl U_t$ we have
		\[
		(\mtcl {J_V})_t|_U\sim_U lA_{\mtcl V_t}|_U.
		\]
		Thus $(\mtcl V_t,\Gamma_t),(\mtcl {J_V})_t$ are what we need.	
		
	\end{proof}
	We now deduce Theorem~\ref{mainthm} from Theorem~\ref{corethm}. The argument is exactly the same as the proof of Theorem 1.7 in \cite{jjz25}. We present the argument for readers' convenience.
	\begin{proof}[Proof of Theorem~\ref{mainthm}]
		\medskip\noindent\emph{Step 1: bounding vertical divisors.}
		By Theorem~\ref{corethm}, there exist a contraction $h:V\to Z$ from a normal variety $V$, an effective integral divisor $J$ on $V$, and a positive integer $l$, depending only on $d,\Phi,v,r,\varepsilon$ with the following properties: Writing $B_V$ for the strict transform of $B$ on $V$, taking a general element $G\in |6dH|$ and setting
		\[
		\Delta_V:=B_V+\frac{1}{2}h^{\ast}G,
		\]
		we have
		\begin{itemize}
			\item $V$ is $\mathbb{Q}$-factorial;
			
			\item the birational map $V\dashrightarrow X$ over $Z$ is an isomorphism in codimension one;
			\item $(V,\Delta_V+J)$ forms a log bounded family; and
			\item $J|_{V_\eta}\sim lA_V|_{V_\eta}$, where $\eta$ is the generic point of $Z$ and $A_V$ is the strict transform of $A$ on $V$.
		\end{itemize}
		By Lemma~5.2 in~\cite{jjz25}, there exist distinct vertical prime divisors $D_1,\ldots,D_k$ over $Z$ and rational numbers $a_i\geq0$ such that $(V,B_V+J+\sum D_i)$ is log bounded and
		\[
		J + \sum a_i D_i \sim_{\mathbb{Q},Z} lA_V.
		\]
		The coefficients of $B_V+J$ belong to a fixed finite set by the preceding construction. Since the $D_i$ are distinct prime divisors appearing with coefficient one in the displayed log bounded divisor, the coefficients of $B_V+J+\sum D_i$ also belong to a finite set. By Lemma~\ref{finitecoeffrefinement} and  Lemma~\ref{hanjiang}, there exist finitely many projective families
		\[
		g_\lambda:(\mathcal V_\lambda,\mathcal B_{\mathcal V_\lambda};
		\mathcal J_\lambda,\mathcal D_{1,\lambda},\ldots,\mathcal D_{k,\lambda})
		\longrightarrow(\mathcal Z_\lambda,\mathcal L_\lambda)
		\longrightarrow\mathcal S_\lambda,
		\qquad 1\leq\lambda\leq M,
		\]
		where
		\begin{itemize}
			\item  $\mathcal S_\lambda$ are  quasi-projective normal varieties; \item $(\mtcl V_{\lambda},\mtcl B_{\mtcl V_{\lambda}})$ are klt;
			\item  $\mtcl{Z}_{\lambda}$ are normal;
			\item  \(
			K_{\mathcal V_\lambda}+\mathcal B_{\mathcal V_\lambda}
			\sim_{\mathbb Q,\mathcal Z_\lambda}0
			\);
			\item $\mtcl{J}_{\lambda}$ are effective integral divisors big over $\mtcl Z_{\lambda}$;
			\item  $\mathcal D_{1,\lambda},\ldots,\mathcal D_{k,\lambda}$ are vertical prime divisors over $\mathcal Z_\lambda$; 
			\item $\mathcal L_\lambda$ are  effective Cartier divisors on $\mathcal Z_\lambda$; and
			
			\item  for every datum
			\[
			h:(V,B_V;J,D_1,\ldots,D_k)\longrightarrow(Z,H)
			\]
			as above, there exist  $1\leq\lambda\leq M$ and a point $s\in\mathcal S_\lambda$ such that
			\[
			\begin{gathered}
				g_{\lambda,s}:((\mathcal V_\lambda)_s,
				(\mathcal B_{\mathcal V_\lambda})_s;
				(\mathcal J_\lambda)_s,(\mathcal D_{1,\lambda})_s,\ldots,
				(\mathcal D_{k,\lambda})_s)
				\longrightarrow((\mathcal Z_\lambda)_s,(\mathcal L_\lambda)_s)
				\\
				\simeq
				\\
				h:(V,B_V;J,D_1,\ldots,D_k)\longrightarrow(Z,H).
			\end{gathered}
			\]
		\end{itemize}

		Choose a sufficiently divisible very ample Cartier divisor $\mathcal H_1$ on $\mathcal Z$.
		
		\medskip\noindent\emph{Step 2: MMP in families.}
		Choose $\mu>0$ sufficiently small so that
		\[
		\left(\mathcal V,\mathcal B_{\mathcal V}
		+\mu\left(\mathcal J+\sum\mathcal D_i\right)\right)
		\]
		is klt. Let $\mathcal G_1\in|6d\mathcal H_1|$ be a general member. Then
		\[
		\left(\mathcal V,\mathcal B_{\mathcal V}
		+\mu\left(\mathcal J+\sum\mathcal D_i\right)
		+\frac12g^{\ast}\mathcal G_1\right)
		\]
		is klt.
		
		Let $\rho:\mathcal W\to\mathcal V$ be a log resolution of
		\[
		\left(\mathcal V,
		\mathcal B_{\mathcal V}+\mathcal J+\sum\mathcal D_i
		+g^{\ast}\mathcal G_1\right).
		\]
		Possibly passing to a stratification on $\mathcal S$, we may assume that $\mtcl S$ is smooth and that $\rho$ is a fiberwise log resolution over $\mtcl S$. Since the support above is independent of the coefficients $a_i$, this stratification is common to all choices of $(a_1,\ldots,a_k)$.
		
		Fix $\mathbf a=(a_1,\ldots,a_k)\in\mathbb Q_{\geq0}^k$ and choose $\mu_{\mathbf a}>0$ so that
		\[
		\mu_{\mathbf a}\left(\mathcal J+\sum a_i\mathcal D_i\right)
		\leq
		\mu\left(\mathcal J+\sum\mathcal D_i\right),
		\]
		where the inequality is coefficientwise. Then
		\[
		\left(\mathcal V,\mathcal B_{\mathcal V}
		+\mu_{\mathbf a}\left(\mathcal J+\sum a_i\mathcal D_i\right)
		+\frac12g^{\ast}\mathcal G_1\right)
		\]
		is klt.
		
		The divisor $\mathcal J+\sum a_i\mathcal D_i$ is big over $\mathcal Z$, so the pair obtained by omitting $\frac12g^{\ast}\mathcal G_1$ has a log canonical model over $\mathcal Z$. By boundedness of the lengths of extremal rays, this is also the log canonical model over $\mathcal S$ of
		\[
		\left(\mathcal V,\mathcal B_{\mathcal V}
		+\mu_{\mathbf a}\left(\mathcal J+\sum a_i\mathcal D_i\right)
		+\frac12g^{\ast}\mathcal G_1\right).
		\]
		Applying Corollary~1.4 of~\cite{hacon2017boundednessmodulivarietiesgeneral} to the induced log-smooth pair on $\mathcal W$, we see that this model induces log canonical models fiberwise.
		
		\medskip\noindent\emph{Step 3: Conclude.}
		
		By Lemma~5.3 in~\cite{jjz25} there exist finitely many rational maps
		\[
		\psi_j:\mathcal V\dashrightarrow\mathcal Y_j/\mathcal Z
		\]
		such that, for every $\mathbf{a}\in \Q^k_{\geq 0}$, there exist some $j$ such that $\psi_j$ is the log canonical model of $\mtcl J+\sum a_i\mtcl D_i$ over $\mtcl Z$.

		Let $\mtcl V_s\simeq V$ and  $\mathbf a=(a_1,\ldots,a_k)$ as in Step 1, and choose $j$ as above. Then $$\psi_{j,s}:V\dasharrow\mathcal Y_{j,s}$$ is the log canonical model over $Z$ of
		\[
		\left(V,B_V+\mu_{\mathbf a}\left(J+\sum a_i D_{i}\right)\right).
		\]
		Since $J+\sum a_i D_{i}\sim_{\mathbb Q,Z}lA_V$, this is also the log canonical model of
		\[
		\left(V,B_V+\mu_{\mathbf a}lA_V\right)
		\]
		over $Z$. Since $A$ is ample over $Z$ and $V\dashrightarrow X$ is an isomorphism in codimension one, $(X,B+\mu_{\mathbf a}lA)$ is also the log canonical model of \(
		\left(V,B_V+\mu_{\mathbf a}\left(J+\sum a_i D_{i}\right)\right).
		\)  Hence $\mathcal Y_{j,s}\simeq X$ over $Z$.
		Thus $(X,\operatorname{Supp}(B+f^{\ast}H))$ belongs to one of the finitely many families above.
	\end{proof}
	
	\appendix
	\section{Compactifiability of the stack of traditional stable minimal models}\label{app:tscompact}
	This appendix gives, for the moduli problem used in the proof, an alternative to the application of Theorem~\ref{tame}. We use the quotient construction in the latest version of~\cite{jiang2025boundednesskltgoodminimal}, with the polarization interpreted as a relative Picard class as in Section~2, and verify the quasi-projectivity needed for the compactification lemma of Kov\'acs--Lieblich~\cite[Lemma~4.2]{kovacs2009boundednessfamiliescanonicallypolarized}, which is a consequence of Kresch's compactification theorem~\cite[Theorem~5.3]{kresch2009geometry}.
	
	In the notation of the present paper, the corrected stack is $\mtcl M_{TS}(d,\Phi,\Gamma,\sigma)$. The notation in~\cite[Appendix~A]{jiang2025boundednesskltgoodminimal} is $\mathfrak{TS}_{\mathrm{klt}}(d,\Phi,\Gamma,\sigma)$. 
	
	\begin{defn}

		Let $f:X\to S$ be projective of pure relative dimension $n\geq1$, fix a relatively very ample divisor $L$, and let $c\geq0$. For an $S$-scheme $T$, write $X_T=X\times_S T$. A relative Mumford divisor $D\subset X_T$ is \emph{$K$-flat over $T$} if, for every $t\in T$ and every finite morphism
		\[
		\pi:X_{T_t}\longrightarrow\mathbb P^n_{T_t},
		\qquad T_t=\operatorname{Spec}\O_{T,t},
		\]
		the divisorial pushforward $\pi_*(D_{T_t})$ is a relative Cartier divisor; see~\cite[Definition~7.1]{book}. Here pullback is divisorial, and pushforward uses the divisorial support construction of~\cite[Section~7.3]{book}. All residue fields are infinite since we work over $\C$.
	\end{defn}
	Define the contravariant functor
	\[
	\mathbf{KDiv}_c(X/S)(T)=
	\left\{\left. D\subset X_T\ \right|\
	\begin{gathered}
		D\text{ is a $K$-flat relative Mumford divisor},\\
		\deg_{L_{\bar t}}D_{\bar t}=c\text{ for every geometric point }\bar t\to T
	\end{gathered}\right\},
	\]
	with transition maps given by divisorial pullback, where $\deg_{L_{\bar t}}D_{\bar t}=(L_{\bar t}^{\,n-1}\cdot[D_{\bar t}])$. By~\cite[Theorem~7.3]{book}, this functor is represented by a separated finite-type $S$-scheme $\operatorname{KDiv}_c(X/S)$, called the \emph{scheme of degree-$c$ $K$-flat relative Mumford divisors}. 
	
	\begin{lem}\label{mdivquasiproj}
		Let $f:X\to S$ and $L$ be as above. If we further assume $S$ is quasi-projective over $\C$, then $\mathrm{KDiv}_c(X/S)$ is also quasi-projective over $\C$.
		
	\end{lem}
	\begin{proof}
		Quasi-projectivity is not part of the statement of~\cite[Theorem~7.3]{book}. It follows, however, from the construction used in its proof.
		
		Choose a closed immersion $X\hookrightarrow\mathbb P^N_S$. By \cite[Theorem~4.76]{book}, the scheme $\mathrm{MDiv}_c(X/S)$ of relative Mumford divisors is quasi-projective over $S$. By \cite[Proposition~7.65 and the proof of Theorem~7.3 in~7.66]{book}, $\mathrm{KDiv}_c(X/S)$ is a full subscheme of $\mathrm{MDiv}_c(X/S)$. Thus it is also quasi-projective over $S$.

	\end{proof}

	\begin{prop}\label{tscompact}
		For the numerical data occurring in this paper, the moduli stack
		\[
		\mtcl M_{TS}(d,\Phi,\Gamma,\sigma)
		\]
		is compactifiable. More precisely, there is an open immersion
		\[
		\mtcl M_{TS}(d,\Phi,\Gamma,\sigma)
		\hookrightarrow
		\overline{\mtcl M}_{TS}
		\]
		into a proper Deligne--Mumford stack of finite type over $\C$.
	\end{prop}
	\begin{proof}
		\medskip\noindent\emph{Step 1: reduce to finitely many numerical strata.}
		We trace the construction in~\cite[Appendix~A.2]{jiang2025boundednesskltgoodminimal}. By~\cite[Lemma~A.9]{jiang2025boundednesskltgoodminimal}, there is a finite disjoint decomposition into open-and-closed substacks
		\[
		\mtcl M_{TS}(d,\Phi,\Gamma,\sigma)
		=
		\bigsqcup_{\alpha}\mtcl M_{\alpha},
		\]
		where $\alpha=(h,n,l,\mathbf c)$ records the Hilbert polynomial, the two dimensions of spaces of sections, and the degrees of the marked boundary divisors. For a fixed $\alpha$, the proof of~\cite[Theorem~A.11]{jiang2025boundednesskltgoodminimal} starts with
		\[
		H=\operatorname{Hilb}_h(\mathbb P^n\times\mathbb P^l)
		\]
		and its open subscheme $H_1$ parametrizing the required geometrically connected, reduced, equidimensional, and geometrically normal fibers. Step~2 constructs
		\[
		H_2=
		\operatorname{MDiv}_{c_1}(X_{H_1}/H_1)
		\times_{H_1}\cdots\times_{H_1}
		\operatorname{MDiv}_{c_m}(X_{H_1}/H_1).
		\]
		The Hilbert scheme $H$ is projective, hence $H_1$ is quasi-projective. Lemma~\ref{mdivquasiproj} shows that every factor in the displayed fiber product is quasi-projective over $H_1$. Therefore $H_2$ is quasi-projective over $\C$.
		
		\medskip\noindent\emph{Step 2: prove that the final parameter scheme is quasi-projective.}
		Steps~3--5 of the same proof successively construct
		\[
		\begin{aligned}
			&H_3\longrightarrow H_2,\qquad H_4\longrightarrow H_3,\qquad
			H_5\subset H_4,\qquad H_6\subset H_5,\\
			&H_7\subset H_6,\qquad H_8\subset H_7,\qquad H_9\subset H_8.
		\end{aligned}
		\]
		Here $H_3\to H_2$ and $H_4\to H_3$ are locally closed partial decompositions, $H_5$ and $H_6$ are locally closed, $H_7$ is open, $H_8$ is locally closed, and $H_9$ is open and closed. These are the precise conclusions recorded in Steps~3--5 of~\cite[Theorem~A.11]{jiang2025boundednesskltgoodminimal}. Jiang also shows that $H_9$ is of finite type. Hence every connected component of $H_9$ is locally closed in $H_2$ and there are only finitely many such components. Thus $H_9$ is quasi-projective.
		
		\medskip\noindent\emph{Step 3: compactify the quotient stacks.}
		Finally, the argument in Step~6 of~\cite[Theorem~A.11]{jiang2025boundednesskltgoodminimal}, applied to local representatives of the relative Picard class, gives the quotient presentation
		\[
		\mtcl M_{\alpha}
		\simeq
		\left[H_9/\bigl(\operatorname{PGL}_{n+1}\times
		\operatorname{PGL}_{l+1}\bigr)\right].
		\]
		The acting group is a linear algebraic group. Moreover, $\mtcl M_{\alpha}$ is a separated Deligne--Mumford stack of finite type because it is an open-and-closed substack of the stack in Theorem~\ref{tsmoduli}. The compactification lemma of Kov\'acs--Lieblich~\cite[Lemma~4.2]{kovacs2009boundednessfamiliescanonicallypolarized}, proved from~\cite[Theorem~5.3]{kresch2009geometry}, therefore gives an open immersion of $\mtcl M_{\alpha}$ into a proper Deligne--Mumford stack $\overline{\mtcl M}_{\alpha}$.
		
		There are only finitely many indices $\alpha$. Thus
		\[
		\overline{\mtcl M}_{TS}:=
		\bigsqcup_{\alpha}\overline{\mtcl M}_{\alpha}
		\]
		is a proper Deligne--Mumford stack of finite type, and the finite disjoint union of the preceding open immersions proves the proposition.
	\end{proof}

	\bibliographystyle{alpha}
	\bibliography{bbb}

\end{document}